\documentclass[11pt,a4paper]{amsart}
\newcommand{\Title}{Algebraic Aspects and Functoriality of the set of Affiliated Operators}%
\newcommand{\ShortTitle}{\Title}

\makeatletter
\newcommand{\raisemath}[1]{\mathpalette{\raisem@th{#1}}}
\newcommand{\raisem@th}[3]{\raisebox{#1}{$#2#3$}}
\makeatother

\newcommand{\AuthorOne}{Indrajit Ghosh}%
\newcommand{\AuthorOneAddr}{%
Statistics and Mathematics Unit, Indian Statistical Institute, 8th Mile, Mysore Road, RVCE Post, Bengaluru and Karnataka - 560 059, India
}%
\newcommand{\AuthorOneEmail}{%
indrajitghosh912@gmail.com
}%

\newcommand{\AuthorTwo}{Soumyashant Nayak}%
\newcommand{\AuthorTwoAddr}{%
Statistics and Mathematics Unit, Indian Statistical Institute, 8th Mile, Mysore Road, RVCE Post, Bengaluru and Karnataka - 560 059, India
}%
\newcommand{\AuthorTwoEmail}{%
soumyashant@isibang.ac.in
}%

\newcommand{\Keywords}{von Neumann Algebras, Affiliated Operators, Unbounded Operators}%

\newcommand{\SubjectClassification}{Primary: 46L10; Secondary: 47C15, 47L60, 46M15, 47B02}%

\newcommand{\pdfLinkColor}{purple}
\newcommand{\pdfUrlColor}{magenta}
\newcommand{\pdfCiteColor}{cyan}
\newcommand{\pdfTitle}{\Title}%
\newcommand{\pdfAuthor}{Indrajit Ghosh and Soumyashant Nayak}%
\newcommand{\pdfSubject}{Mathematics}%
\newcommand{\pdfKeywords}{\Keywords}%
\newcommand{\pdfCreator}{TeX Live 2020}%
\newcommand{\pdfCreationDate}{\today}%
\newcommand{\pdfColorLink}{true}%

\usepackage[utf8]{inputenc}%
\usepackage[T1]{fontenc}%
\usepackage{lmodern}
\usepackage[top=1.15in, bottom=1.15in, left=1.2in, right=1.2in]{geometry}
\usepackage{amsmath,amssymb}
\usepackage{amsthm}
\usepackage{mathtools}
\usepackage{mathrsfs} 
\usepackage{xfrac} 
\usepackage{dsfont} 
\usepackage{array}
\usepackage{verbatim}
\usepackage{graphicx}
\usepackage{enumitem} 
\usepackage{lipsum}%
\usepackage{tikz-cd}
\usepackage{hyperref}
\hypersetup{
	pdftitle={\pdfTitle},
	pdfauthor={\pdfAuthor},
	pdfsubject={\pdfSubject},
	pdfcreationdate={\pdfCreationDate},
	pdfcreator={\pdfCreator},
	pdfkeywords={\pdfKeywords},
	colorlinks=\pdfColorLink,
	linkcolor={\pdfLinkColor},
	urlcolor=\pdfUrlColor,
	citecolor=\pdfCiteColor,
	pdfpagemode=UseOutlines,
}
	
\theoremstyle{plain} 
\begingroup 
\newtheorem{thm}{Theorem}[section] 
\newtheorem{cor}[thm]{Corollary} 
\newtheorem{lem}[thm]{Lemma} 
\newtheorem{prop}[thm]{Proposition} 
\endgroup

\theoremstyle{definition}
\newtheorem{definition}[thm]{Definition} 
\theoremstyle{remark}
\newtheorem{remark}[thm]{Remark} 
\newtheorem{example}[thm]{Example} 

\newtheoremstyle{ser}
{8pt}
{8pt}
{\it}
{}
{\sf\bfseries}
{:}
{6mm}
{}

\newtheoremstyle{serr}
{8pt}
{8pt}
{\normalfont}
{}
{\sf}
{.}
{6mm}
{}

\theoremstyle{ser}

\theoremstyle{serr}

\theoremstyle{ser}

\numberwithin{equation}{section}

     		\newcommand{\kA}{\mathscr{A}}
     \newcommand{\sB}{\mathcal B}		
     \newcommand{\sC}{\mathcal C}		\newcommand{\kC}{\mathscr{C}}
     \newcommand{\sD}{\mathcal D}		
     		
     \newcommand{\sF}{\mathcal F}		
     		
     \newcommand{\sH}{\mathcal H}		
     		
\newcommand{\fK}{\mathfrak K}     \newcommand{\sK}{\mathcal K}		
     		
\newcommand{\fL}{\mathfrak L}     \newcommand{\sL}{\mathcal L}		
     		\newcommand{\kM}{\mathscr{M}}
     		\newcommand{\kN}{\mathscr{N}}

\newcommand{\fS}{\mathfrak S}     \newcommand{\sS}{\mathcal S}		
\newcommand{\fT}{\mathfrak T}     		
     		
     \newcommand{\sV}{\mathcal V}		\newcommand{\kW}{\mathscr{W}}
     \newcommand{\sW}{\mathcal W}

\def\XXint#1#2#3{{\setbox0=\hbox{$#1{#2#3}{\int}$ }
		\vcenter{\hbox{$#2#3$ }}\kern-.6\wd0}}

\newcommand{\R}{\mathbb{R}}

\newcommand{\C}{\mathbb{C}}

\newcommand{\N}{\mathbb{N}}

\newcommand{\tn}[1]{\textnormal{#1}}

\DeclareRobustCommand{\rchi}{{\mathpalette\irchi\relax}}
\newcommand{\irchi}[2]{\raisebox{\depth}{$#1\chi$}} 

\newcommand{\Bh}{\mathcal{B}(\mathcal{H})} 

\newcommand{\ran}{\textnormal{\textrm{ran}}}
\newcommand{\nul}{\textnormal{\textrm{null}}}
\newcommand{\dom}{\textnormal{\textrm{dom}}}

\newcommand{\aff}{\textrm{aff}}
\newcommand{\afm}{\mathscr M _{\tn\aff}}
\newcommand{\afn}{\mathscr N _{\tn\aff}}

\newcommand{\ipdt}[2]{\left\langle {#1, #2} \right\rangle}

\newcommand{\raR}{\textbf{R}}
\newcommand{\nulN}{\textbf{N}}
\newcommand{\krein}[1]{{#1}^{\textbf{Kr}}}
\newcommand{\fried}[1]{{#1}^{\textbf{Fr}}}

\newcommand{\graph}[1]{\text{Graph}\left( #1\right)}

\newcommand{\invran}{\mathrel{\overset{\makebox[0pt]{\mbox{\normalfont\tiny\sffamily inv}}}{\circ}}}

\newcommand{\mvnaff}[1]{#1_{\aff}^{\text{MvN}}}
\newcommand{\affc}[1]{#1_{\text{aff}}^{c}}
\newcommand{\affs}[1]{\text{Aff}_{s}\left( #1 \right)}

\begin{document}%
%


\title[\MakeUppercase\ShortTitle]{\MakeUppercase \Title}
\author{\AuthorOne}%
\address[\AuthorOne]{\AuthorOneAddr}%
\email{\href{mailto:\AuthorOneEmail}{\AuthorOneEmail}}%

\author{\AuthorTwo}%
\address[\AuthorTwo]{\AuthorTwoAddr}%
\email{\href{mailto:\AuthorTwoEmail}{\AuthorTwoEmail}}%

\date{}%
\subjclass[2020]{\SubjectClassification}%



\begin{abstract}%
    \label{sec:abstract}%
In this article, we aim to provide a satisfactory algebraic description of the set of affiliated operators for von Neumann algebras. Let $\mathscr{M}$ be a von Neumann algebra acting on a Hilbert space $\mathcal{H}$, and let $\mathscr{M}_{\text{aff}}$ denote the set of unbounded operators of the form $T = AB^{\dagger}$ for $A, B \in \mathscr{M}$ with $\text{null} (B) \subseteq \text{null} (A)$, where $(\cdot)^{\dagger}$ denotes the Kaufman inverse. We show that $\mathscr{M}_{\text{aff}}$ is closed under sum, product, Kaufman-inverse and adjoint, and has the structure of a {\it (right) near-semiring}; Moreover, the above quotient representation of an operator in $\mathscr{M}_{\text{aff}}$ is essentially unique. Thus we may view $\mathscr{M}_{\text{aff}}$ as the multiplicative monoid of unbounded operators on $\mathcal{H}$ generated by $\mathscr{M}$ and $\mathscr{M} ^{\dagger}$. We further show that our definition of affiliation, as reflected in $\mathscr{M}_{\text{aff}}$, subsumes the traditional one. Let $\Phi$ be a unital normal $*$-homomorphism between represented von Neumann algebras $(\mathscr{M}; \mathcal{H})$ and $(\mathscr{N}; \mathcal{K})$. Using the quotient representation, we obtain a canonical extension of $\Phi$ to a mapping $\Phi_{\text{aff}} : \mathscr{M}_{\text{aff}} \to \mathscr{N}_{\text{aff}}$ which is a near-semiring homomorphism that respects Kaufman-inverse and adjoint; in addition, $\Phi_{\text{aff}}$ respects Murray-von Neumann affiliation of operators and also respects strong-sum and strong-product. Thus $\mathscr{M}_{\text{aff}}$ is intrinsically associated with $\mathscr{M}$ and transforms functorially as we change representations of $\mathscr{M}$. Furthermore, $\Phi_{\text{aff}}$ preserves operator properties such as being symmetric, or positive, or accretive, or sectorial, or self-adjoint, or normal, and also preserves the Friedrichs and Krein-von Neumann extensions of densely-defined closed positive operators. As a proof of concept, we transfer some well-known results about closed unbounded operators to the setting of closed affiliated operators for properly infinite von Neumann algebras, via `abstract nonsense'.
\end{abstract}

\maketitle%
\thispagestyle{empty}%

\tableofcontents

\section{Introduction}%
\label{sec:introduction}%

In von Neumann's formulation of the mathematical foundations of quantum mechanics (see \cite{vNeuman_quantum}), the necessity of studying unbounded operators becomes apparent at the outset; for instance,  the momentum observable corresponds to an unbounded self-adjoint operator on a separable Hilbert space. The process of differentiation already points towards the naturality and importance of studying unbounded operators. Motivated by the closed graph theorem for bounded operators and with a view towards allowing vestiges of analysis to enter into the picture (as opposed to considering arbitrary linear operators), unbounded operators with closed graphs have been extensively studied. Such operators are known as {\it closed operators} and the process of differentiation (say, on the Sobolev space  $H^1(\R) \subseteq L^2(\R)$) offers a natural example.

In the literature, the relationship between closed operators and bounded operators is generally described in terms of loose analogies, similar to how we did in the previous paragraph. A geometric recipe for converting closed operators into bounded operators is provided by projections onto their graphs, thus making their study amenable to bounded-operator techniques; this viewpoint may be attributed to Stone (see \cite{stone51}). In this article, we provide a novel, simple, and complete algebraic characterization of closed operators in terms of bounded operators, which is roughly sketched out in the next paragraph. In fact, in Theorem \ref{thm:quotient} we establish such a result in the more general setting of closed {\it affiliated operators} for a von Neumann algebra. 

For a Hilbert space $\sH$, by the phrase ``unbounded operator on $\sH$'' we mean a not-necessarily-bounded linear operator whose domain and range lie in $\sH$. The Kaufman inverse of an unbounded operator on $\sH$ (see Definition \ref{def:kaufman_inv}) is denoted by $(-)^\dagger$ and plays a central role in our discussion. The set of bounded everywhere-defined operators on $\sH$ is denoted by $\sB(\sH)$. For $A \in \sB(\sH)$, the closed graph theorem tells us that $A$ is a closed operator. Moreover, the restriction of $A$ to a closed subspace $\sV$ of $\sH$ is a closed (and bounded) operator, and from Remark \ref{rem:kauf_inv_1}, we have $A|_{\sV} = AE_{\sV}^{\dagger}$ where $E_{\sV}$ denotes the orthogonal projection from $\sH$ onto $\sV$. From Lemma \ref{lem:kauf_inv_closed}, we note that the Kaufman inverse of $A$, $A^{\dagger}$, is a closed operator. It is an elementary fact (see Lemma \ref{lem:unbdd_0}) that $TA'$ is closed whenever $T$ is closed and $A'$ is closed and bounded. Thus for bounded operators $A, B, E \in \sB(\sH)$ with $E$ a projection, setting $T=B^{\dagger}$ and $A' = AE^{\dagger}$, we see that the operator $B^{\dagger}AE^{\dagger}$ is a closed operator on $\sH$. In Theorem \ref{thm:quotient}-(ii), we prove that the converse holds, by demonstrating that every closed operator on $\sH$ has a representation of the form $P^{-1} AE^{\dagger}$ for $A, P, E \in \sB(\sH)$ with $P$ being a positive operator with dense range (and hence, trivial nullspace) and $E$ a projection. Our main purpose behind this investigation is to facilitate the careful examination of the notion of affiliated operators for von Neumann algebras, and to gain an understanding of their elusive algebraic structure.

 Given a von Neumann algebra $\kM$ acting on the Hilbert space $\sH$, Murray and von Neumann found it a powerful tool (see \cite{MvN-I}) to consider those closed operators on $\sH$ that are `affiliated' with $\kM$, meaning that $U^*TU = T$ for all unitary operators $U$ in $\kM'$, the commutant of $\kM$ relative to $\sB(\sH)$. For reasons that will become apparent later, we call these operators MvN-affiliated operators instead of simply calling them affiliated operators. We denote the set of MvN-affiliated operators by $\afm^{\text{MvN}}$. Since the commutant of $\sB(\sH)$ is trivial, note that $\sB(\sH)_{\aff}^{\text{MvN}}$ is nothing but the set of closed operators on $\sH$. It is easily surmised that the above definition of $\afm^{\text{MvN}}$ is inspired by the double commutant theorem which asserts that $\mathscr{M} = (\mathscr{M}')'$. On one hand, this appears to suggest an intrinsic association of these operators with $\mathscr{M}$, independent of its representation on Hilbert space (hence the term `affiliated'). On the other hand, since unbounded operators come equipped with a domain (which is a subspace of $\mathcal{H}$) and the definition of MvN-affiliated operators involves the commutant of $\mathscr{M}$ relative to $\mathcal{B}(\mathcal{H})$, the Hilbert space $\mathcal{H}$ appears to be omnipresent.  

In \cite[pg. 189]{skau_paper},  it is noted that if $\Phi : \mathscr{M} \to \mathscr{N}$ is a $*$-isomorphism between von Neumann algebras $\mathscr{M}$ and $\mathscr{N}$, then there is a one-to-one correspondence between $\afm^{\text{MvN}}$ and $\afn^{\text{MvN}}$ and a recipe for the identification is provided via the polar decomposition theorem for MvN-affiliated operators, and the spectral theorem for unbounded self-adjoint operators.  Although this serves as a loose indication of the intrinsic nature of the set of affiliated operators, the following questions need to be addressed for a complete and satisfactory picture.
\begin{enumerate}
    \item Is the identification given in \cite[pg. 189]{skau_paper} related to the underlying isomorphism $\Phi$ in a canonical manner?
    \item Instead of a $*$-isomorphism, if we have a unital normal {\it homomorphism} $\Phi : \mathscr{M} \to \mathscr{N}$, does $\Phi$ extend to a map from $\afm^{\text{MvN}}$ to $\afn^{\text{MvN}}$ in a canonical manner? In other words, is the construction, $\mathscr{M} \mapsto \afm^{\text{MvN}}$, functorial in the appropriate sense?
\end{enumerate}
In this paper, we answer the above two questions in the affirmative. Before we outline our approach, we point out the subtleties in the questions and why claims of `functoriality' in the literature (which only involve isomorphisms) may be flawed. 

Let $\mathscr{W}^*$-Alg denote the category of von Neumann algebras with morphisms as unital normal $*$-homomorphisms.  Let $\sS (\mathscr{M})$ ($\sL (\mathscr{M})$, respectively) denote the $*$-algebra of measurable operators (locally measurable operators, respectively) which are subsets of $\afm^{\text{MvN}}$ (see Definitions \ref{def:meas_op}, \ref{def:loc_mes_op}). In Remark \ref{rem:mes_op}, we note that the constructions $\mathscr{M} \mapsto \sS(\mathscr{M})$, $\mathscr{M} \mapsto \sL(\mathscr{M})$, are {\bf not} functorial on $\mathscr{W}^*$-Alg, though for an isomorphism $\mathscr{M} \cong \mathscr{N}$ in $\mathscr{W}^*$-Alg, we do have isomorphisms $\sS (\mathscr{M}) \cong \sS (\mathscr{N})$ and $\sL(\mathscr{M}) \cong \sL (\mathscr{N})$ as $*$-algebras in a canonical manner. In other words, the set of measurable operators (or, locally measurable operators) for $\kM$ may not be well-behaved with respect to arbitrary normal $*$-representations of $\kM$ but is well-behaved with respect to {\it faithful} normal $*$-representations of $\kM$.

When $\kM$ is a finite von Neumann algebra, in \cite{berberian}, \cite{roos}, the set of densely-defined operators in $\afm ^{\text{MvN}}$ is described as the {\it maximal quotient ring} of $\kM$. In fact, it may be viewed as the Ore localization of $\kM$ with respect to the multiplicative subset of non-zero-divisors (see \cite{handelman}), so that every non-zero-divisor in $\kM$ has an inverse in $\afm^{\text{MvN}}$. In \cite{MvN-I}, it is observed that the notions of strong-sum and strong-product (see Definition \ref{def:str_sum_pdt}) define a ring structure on the set of densely-defined operators in $\afm ^{\text{MvN}}$, and with the adjoint operation, it becomes a $*$-algebra. These $*$-algebras are also known as {\it Murray-von Neumann algebras} in the literature (see \cite{kadison_liu_2014}, \cite{Nayak_2021}). In \cite{Nayak_2021}, topological and order-theoretical aspects of Murray-von Neumann algebras are further studied and it is shown that, the $*$-algebra of densely-defined operators in $\afm^{\text{MvN}}$ for a finite von Neumann algebra $\kM$ gives a functorial construction from the category of finite von Neumann algebras to the category of unital ordered complex topological $*$-algebras. As evidenced above, when $\kM$ is a finite von Neumann algebra, much is known and has been said about the algebraic/topological structure of the set of densely-defined operators in $\afm ^{\text{MvN}}$.

Unfortunately, the same cannot be said for arbitrary von Neumann algebras, and so far in the literature, $\afm^{\text{MvN}}$ has been treated merely as a set, with little mention or investigation of the algebraic structure it may possess. A major issue is the peculiar behaviour of domains of unbounded operators. For example, the sum of two closed operators need not be closed or even pre-closed, and the sum of two densely-defined operators need not be densely-defined (see \cite[pg. 134]{messirdi_almost_closed}). Note that for operators affiliated with {\it finite} von Neumann algebras, the afore-mentioned domain issues are more manageable.

Clearly, any attempt at understanding the functorial nature of affiliation must begin with the exploration of the functorial nature of the set of domains of affiliated operators. With the Douglas factorization lemma (see \cite{douglas_lemma_paper}, \cite[Theorem 2.1]{nayak_douglas_vNa}) as the main tool, in \S \ref{sec:aff_subspaces} we establish the functorial nature of the lattice of operator ranges, which we denote by $\affs{\kM}$. As a generalization of a result of Fillmore and Williams (see \cite[Theorem 1.1]{fillmore-williams}), in Theorem \ref{thm:dom_ran} we show that $\affs{\kM}$ is precisely the set of domains of MvN-affiliated operators for $\kM$.

In our quest to identify the appropriate algebraic structure on the set of affiliated operators, we first expand the scope of the definition of affiliation. 

\begin{definition}
\label{def:affiliation}
Let $\kM$ be a von Neumann algebra acting on the Hilbert space $\sH$. An unbounded operator $T$ on $\sH$ is said to be {\it affiliated} with $\mathscr{M}$ if $T = AB^\dagger$ for some $A, B\in \kM$ with $\nul(B) \subseteq \nul(A)$. The set of operators affiliated with $\kM$ is denoted by $\afm$. 
\end{definition}
Note that the above definition makes no mention of the commutant of $\kM$. For an intuitive interpretation of affiliated operators as `quotients' of operators in $\kM$, the reader may refer to the discussion in the beginning of \S \ref{sec:aff_operators}. Before we point out the connection between our definition of affiliation reflected in $\afm$ and the notion of MvN-affiliation, we take a short detour. 

The operators of the form $AB^{\dagger}$ have been studied under various guises in the literature. They are called ``op\'{e}rateur J uniforme'' by Dixmier in \cite{dixmier1949}, and ``semi-closed operators'' by Kaufman in  \cite{kaufman_semi_closed}. In \cite{kaufman_semi_closed}, it is shown that the set of semi-closed operators on $\sH$ is the smallest multiplicative set generated by the set of closed operators on $\sH$; moreover, it is also  closed under operator addition or sum. In \cite{izumino_quotient_bdd}, Izumino reproves these results (and more) by providing concrete formulae for sums and products of semi-closed operators. In Theorem \ref{thm:m_aff_monoid}, we generalize these results of Kaufman and Izumino to the setting of von Neumann algebras by showing that $\afm$ is closed under the binary operations, {\it sum} and {\it product}. Furthermore, in Theorems \ref{thm:m_aff_is_kau_closed} and \ref{thm:adj_of_aff_opes}, we observe that $\afm$ is closed under the unary operations, {\it Kaufman-inverse} and {\it adjoint}. Note that our notion of adjoint (see Definition \ref{def:adj_not_dd_op}) is for arbitrary unbounded operators on $\sH$, which need not be densely-defined. We equip $\afm$ with the algebraic structure afforded by the above four algebraic operations, and note that $(\afm; +, \cdot)$ is a {\it right near-semiring} (see Definition \ref{def:near-semiring} and Remark \ref{rem:distributivity_operators}). It is worthwhile to note that our approach not only generalizes the results in \cite{kaufman_semi_closed} and \cite{izumino_quotient_bdd} but also brings substantial conceptual and notational clarity, even in the particular case of $\kM= \sB(\sH)$.

In \cite[Theorem 1]{kaufman}, it is shown that every closed operator on $\sH$ is of the form $AB^{\dagger}$ where $A, B \in \sB (\sH)$; in other words, every closed operator is semi-closed. When $\sH$ is infinite-dimensional, the converse is not generally true as there are plenty of semi-closed operators on $\sH$ that are not closed.  We generalize \cite[Theorem 1]{kaufman} to $\afm$ (see proof of Theorem \ref{thm:quotient}-(i)). In fact, we go a step further in Theorem \ref{thm:quotient}-(ii) by showing that an operator in $\afm$ is closed if and only if it is of the form $P^{-1} AE^{\dagger}$ for $A, P, E \in \kM$ with $P$ a one-to-one positive operator and $E$ a projection. In the particular case of $\kM = \sB(\sH)$, this reduces to the algebraic characterization of closed operators in terms of bounded operators, which we discussed previously in the third paragraph of this section. We note that $\afm^{\text{MvN}}$ is precisely the set of closed operators in $\afm$ (see Theorem \ref{thm:quotient}-(i)); thus our definition of affiliation subsumes the traditional definition of affiliation in the sense of Murray and von Neumann.

In Theorem \ref{thm:uniqueness_of_quotients}, we show that the quotient representation of an operator in $\afm$ is essentially unique. This helps us to see that there is a {\it unique} extension (see Theorem \ref{thm:uniqueness_phi_aff}) of the mapping $\Phi : \kM \to \kN$ to a mapping $\Phi_{\aff} : \afm \to \afn$ (see Definition \ref{defn:Phi_aff}) which respects sum, product, Kaufman-inverse, and adjoint; in fact, for the above uniqueness result, it suffices to assume that the extension respects product and Kaufman-inverse. In order to address the key motivation behind our study, we must answer the following question: In what sense is the extension $\Phi_{\aff}$ canonical? From an algebraic perspective, the answer is clear from Theorem \ref{thm:uniqueness_phi_aff}. For the sake of applications, often our interest is in $\afm ^{\text{MvN}}$ rather than $\afm$. In \S \ref{subsec:closed_functoriality}, we make an effort to convince the reader that the restriction of $\Phi_{\aff}$ to $\afm^{\text{MvN}}$ is the canonical extension of $\Phi$ to $\afm^{\text{MvN}}$, by viewing it from various natural perspectives.

In Theorem \ref{thm:phi_pres_closed}, we show that the extension $\Phi_{\aff} : \afm \to \afn$ maps closed operators to closed operators, so that if necessary, we may consider (by restriction) $\Phi_{\aff}$ as a mapping from $\afm^{\text{MvN}}$ to $\afn^{\text{MvN}}$. Furthermore, $\Phi_{\aff}$ maps densely-defined operators in $\afm^{\text{MvN}}$ to densely-defined operators in $\afn^{\text{MvN}}$. In fact, in Corollary \ref{cor:phi_pres_st_sum_pdt} and Theorem \ref{thm:uniqueness_phi_aff_2}, we observe that $\Phi_{\aff}$ is the {\it unique} extension of $\Phi$ to $\afm^{\text{MvN}}$ that preserves strong-sum, strong-product, and adjoint.

As mentioned earlier, Stone's idea of characteristic projection matrices provides a geometric perspective by converting closed operators on $\sH$ into projections onto closed subspaces of $\sH \oplus \sH$. In Lemma \ref{lem:char_aff}, we show that if $T \in \afm^{\text{MvN}}$, then its characteristic projection matrix, $\rchi(T)$, lies in $M_2(\kM)$. Since $\Phi$ naturally induces a map from $M_2(\kM)$ to $M_2(\kN)$ via $\Phi_{(2)} := \Phi \otimes I_2$ which sends projections to projections, this suggests another recipe for an extension of $\Phi$ to $\afm^{\text{MvN}}$. In Theorem \ref{thm:phi_aff_geom}, we show that in fact $\Phi_{\aff} = \rchi^{-1} \circ \Phi_{(2)} \circ \rchi$ on $\afm^{\text{MvN}}$. At this point, we hope that we have been able to persuade the reader that our definition of affiliation provides significant algebraic simplicity without compromising the traditional usages of the term.

In \S \ref{sec:misc_alg_props}, we show that $\Phi_{\aff}$ preserves operator properties such as being symmetric, positive, accretive, sectorial, self-adjoint, normal. This is achieved by re-interpreting these properties in terms of bounded operators appearing in their respective quotient representations (for example, see Proposition \ref{prop:phi_num_range}). In \S \ref{sec:krein_friedrichs}, we observe that $\Phi_{\textrm{aff}}$ maps the Friedrichs and Krein-von Neumann extensions of a densely-defined closed positive operator $S$ in $\afm$ to their respective counterparts for $\Phi_{\textrm{aff}}(S)$ in $\afn$. En route, a user-friendly crash-course on Krein's extension theory (with a view towards functoriality) is also provided. As a proof of concept, in \S \ref{sec:applications}, we transfer some well-known results about closed unbounded operators to the setting of affiliated operators for properly infinite von Neumann algebras, via `abstract nonsense'.

We have strived to make our account as self-contained as possible.  We anticipate simplifications of many existing results in the literature with the help of the language developed herein, some examples of which will be explored in forthcoming work.

\section{Preliminaries}%

\label{sec:prelim}%

This section is a potpourri of notations, concepts, well-known results, and some elementary lemmas that assist us throughout this article. Our primary reference for unbounded operators is \cite{konrad_unbdd}, and for the general theory of von Neumann algebras, we use \cite{kr-I, kr-II}. In \S \ref{subsec:near_semiring}, we briefly discuss near-semirings.

\vspace{0.3cm}\noindent{}\textbf{Notations.}
\label{paper_notations}
    All of our Hilbert spaces are over the field $\C$. We shall denote von-Neumann algebras by $\kM, \kN, \dots$, Hilbert spaces by $\sH, \sK, \dots$, operators on Hilbert spaces by $S, T, A, B, \dots$ and vectors in a Hilbert space by $x, y, z, \dots$; The zero-dimensional subspace of $\sH$ is denoted by $\{ 0 \}_{\sH}$.\ The set of bounded operators on a Hilbert space $\sH$ is denoted by $\sB(\sH)$, and the set of closed operators (see Definition \ref{def:closed_op}) on $\sH$ is denoted by $\sC(\sH)$. When $\kM$ is a von Neumann algebra acting on the Hilbert space $\sH$, we denote it by $(\kM; \sH)$ and call it a represented von Neumann algebra. The lattice of projections in $\kM$ will be denoted by $\kM_{\text{proj}}$ and the set of all self-adjoint operators in $\kM$ will be denoted by $\kM_{\text{sa}}$. For $k \in \N$, the unital normal homomorphism from $M_k(\mathscr{M})$ to $M_k(\mathscr{N})$ given by $$\left[A_{ij}\right]_{1 \le i, j \le k} \mapsto \left[\Phi(A_{ij})\right]_{1 \le i, j \le k},$$ is denoted by $\Phi_{(k)}$. For a bounded operator $T$ on a Hilbert space, we use the notation $|T| := (T^*T)^{\frac{1}{2}}$. We will frequently use this in the context of $T^*$; so it helps to keep in mind that $|T^*| =  (TT^*)^{\frac{1}{2}}$.

\subsection{Unbounded Operators}
For a Hilbert space $\sH$, note that $\mathcal{B}(\oplus_{i=1}^k \sH)$ and $M_k \big( \Bh \big)$ are isomorphic as von Neumann algebras (see \cite[\S 2.6]{kr-I} for more details). We often use this identification to view $M_k \big( \Bh \big)$ as acting on $\oplus_{i=1}^k \sH$. For a represented von Neumann algebra $(\kM; \sH)$, we usually consider $M_k(\kM)$ as the represented von Neumann algebra $\big( M_k(\kM); \oplus_{i=1}^k \sH \big)$ in the above sense. For bounded (everywhere-defined) operators $A', B'$ on a Hilbert space $\sH'$, we use the notation, 
\[
\fT_{A', B'} :=  \begin{bmatrix}
A' & -B'\\
0 & 0
\end{bmatrix},
\fS_{A', B'} := 
\begin{bmatrix}
B' & 0\\
A' & 0
\end{bmatrix} \in M_2 \big( \sB(\sH') \big).
\]

Below we recall some basic definitions and results about linear operators on a Hilbert space $\sH$. Any linear operator $T$ with domain and range in $\sH$ is called an \emph{unbounded operator} on $\sH$. We denote the domain of definition of an unbounded linear operator $T$ on a Hilbert space $\sH$ by $\dom(T)$, which potentially constitutes of a smaller subset of $\mathcal{H}$. We emphasize that by the term ``unbounded'', we mean ``not necessarily bounded'' rather than ``not bounded''. The \emph{graph} of $T$ is defined as,
\[
	\graph{T} := \Big\{ (x, Tx)~:~x\in\dom(T) \Big\} \subseteq \sH\oplus\sH.
\]

\begin{definition}[Closed Operator]
\label{def:closed_op}
A linear operator $T$ acting on $\sH$ is said to be closed if $\graph{T}$ is a closed subset of $\sH\oplus\sH$. We say that $T$ is \emph{pre-closed} or \emph{closable} if the closure of $\graph{T}$ is the graph of an operator. We denote this operator by $\overline{T}$ and call it the {\it closure} of $T$. When $T$ is pre-closed, note that $\graph{\overline{T}} = \overline{\graph{T}}$.
\end{definition}

\begin{definition}
For unbounded operators $T, T_0$ on $\sH$, we say that $T_0$ is an \emph{extension} of $T$, symbolically written as $T\subseteq T_0$, when the following conditions hold:
\begin{enumerate}
	\item $\dom(T)\subseteq \dom(T_0)$, and
	\item $Tx = T_0x$ for all $x\in \dom(T)$.
\end{enumerate}
If $T$ is pre-closed, note that the closure of $T$, $\overline{T}$, is the smallest closed extension of $T$. 
\end{definition}

\begin{definition}
\label{def:str_sum_pdt}
For a pair of unbounded operators $S, T$ on $\sH$, the operators $S+T$ and $ST$ are defined as follows:
\begin{align*}
	\dom(S + T)&:=\dom(S)\cap \dom(T), \\
	(S+T)x&:=Sx+Tx, \text{ for all }x\in \dom(S+T);\\
	\dom(ST)&:=\{x\in \dom(T)~:~Tx \in \dom(S)\},\\
	(ST)x&:=S(Tx),\text{ for all }x\in \dom(ST).
\end{align*}

If $S, T$ are closed operators on $\sH$ such that $S+T$ ($ST$, respectively) is pre-closed, then its closure $\overline{S+T}$ ($\overline{ST}$, respectively) is called the {\it strong-sum} ({\it strong-product}, respectively) of $S$ and $T$, and is denoted by $S \;\hat{+}\; T$ ($S \;\hat{\cdot}\; T$, respectively).
\end{definition}

\begin{remark}
\label{rem:distributivity_operators}
    Let $\sH$ be a Hilbert space, and $R, S, T$ be unbounded operators on $\sH$. Then it is easily verified that     \[
    R + (S + T) = (R + S) + T \, \text{ and }\, R(ST) = (RS) T.
    \]
    In other words, the sum and product are associative binary operations on the set of unbounded operators on $\sH$. But only right-distributivity holds in general:
    \[
    (S + T)R = SR + TR , \;\; R(S+T) \supseteq RS + RT.
    \]

When $\sH$ is infinite-dimensional, by a result of Neumark (see \cite{neumark_square}, \cite{konrad_1983}, \cite{chernoff_paul}), it is known that there is a densely-defined closed symmetric operator $T$ on $\sH$ such that $\dom (T^2) = \{ 0 \}_{\sH}$. It is straightforward to see that $T^2 + (-T^2) = 0|_{ \{0\}_{\sH}}$, whereas $T \left( T+(-T) \right) = 0|_{\dom(T)}$. Thus left-distributivity does {\bf not} hold in general.
\end{remark}

\begin{remark}
\label{rem:prod_ext}
Let $S, S_0,T, T_0$ be unbounded operators on $\sH$ with $S \subseteq S_0, T \subseteq T_0$, then it is straightforward to verify that, 
$$ST \subseteq S_0 T_0,\; TS \subseteq T_0 S_0, \;
S + T \subseteq S_0 + T_0.$$
\end{remark}

\begin{definition}
\label{def:ran_and_null_proj}
Let $T$ be an unbounded operator acting on $\sH$. We define, $$\nul (T) := \left\{ x \in \dom(T) : Tx = 0 \right\},\; \ran(T) := \left\{ Tx : x \in \dom(T) \right\}.$$
Since $\{ 0 \}_{\sH} \subseteq \nul (T)$, note that the nullspace of $T$ is non-empty. The projection onto the closed subspace $\overline{\ran(T)} \subseteq \sH$ is said to be the \emph{range projection} of $T$, and is denoted by $\raR(T)$. The projection onto the closed subspace $\overline{\nul(T)} \subseteq \sH$ is said to  be the \emph{null projection} of $T$, and is denoted by $\nulN(T)$.
\end{definition}

In \cite{kaufman}, in the course of his investigation of quotient representations of closed operators, Kaufman defines an `inverse' for arbitrary bounded operators which need not be one-to-one. Below we extend this operation to arbitrary unbounded operators and call it the {\it Kaufman inverse} in his honour. 

\begin{definition}[Kaufman Inverse]
\label{def:kaufman_inv}
Let $\sH$ be a Hilbert space and $T$ be an unbounded operator on $\sH$. Consider the unbounded operator, $T^{\dagger}$, on $\sH$, with $\dom(T^{\dagger}) := \ran(T)$, which is defined as follows, 
\[
    Tx \mapsto \big( I-\nulN(T) \big) x, \; (\text{for }x \in \dom(T)).
\]
This is clearly well-defined and we call $T^{\dagger}$ the {\it Kaufman inverse} of $T$. When $T$ is one-to-one, that is, $\nul (T) = \{ 0 \}_{\sH}$, it is easy to see that $T^{\dagger} = T^{-1}$ on $\ran(T)$. 

Clearly, $$\ran(T^{\dagger}) = \big( I - \nulN(T) \big) \dom(T) \subseteq \overline{\dom(T)} \cap \nul (T) ^{\perp}.$$ It is left to the reader to verify that $\ran(T^{\dagger})$ is a dense subspace of $\overline{\dom(T)} \cap \nul (T) ^{\perp}$. Thus $T^{\dagger} : \ran(T) \to \overline{\dom(T)} \cap \nul (T) ^{\perp}$ is a one-to-one mapping with dense range.
\end{definition}

\begin{remark}
\label{rem:kauf_inv_1}
Since a bounded operator $A \in \sB(\sH)$ is everywhere-defined, that is, $\dom(A) = \sH$, $A^{\dagger}$ gives a bijection between $\ran(A)$ and $\nul(A)^\perp$. Furthermore for a projection $E \in \sB (\sH)$, it is straightforward to see that $E^{\dagger}$ is the restriction of the identity operator to $\ran(E)$. It is helpful to keep these two observations in mind for later use in this paper (see Lemma \ref{lem:kauf_inv_closed}, Proposition \ref{prop:t_bar}, Theorem \ref{thm:uniqueness_phi_aff}).

The restriction of $0_{\sH}$ to $\{ 0 \}_{\sH}$ is said to be the {\it trivial affiliated operator}, and is equal to $0_{\sH}^{\dagger}$.
\end{remark}

Although the following lemma about closed operators is well-known and standard, we include its proof with a view towards Lemma \ref{lem:unbdd} which generalizes it to the context of affiliated operators.

\begin{lem}
\label{lem:unbdd_0}
\textsl{
Let $\sH$ be a Hilbert space. Let $T, A$ be closed operators on $\sH$ such that $\dom(A)$ is a closed subspace of $\sH$. Then $TA$ is a closed operator on $\sH$.
}
\end{lem}
\begin{proof}
By the closed graph theorem, the operator $A : \dom(A) \to \sH$ is bounded. In order to show that $TA$ is closed, it suffices to show that whenever $x_n \to 0$ in $\dom (TA)$ and $(TA)x_n \to y$, we have $y=0$. 

Let $x_n \to 0$ in $\dom(TA)$ with $(TA)x_n \to y$. Since $A$ is bounded, we have $Ax_n \to 0$ in $\dom(T)$. As $T$ is closed and $T(Ax_n) \to y$, we see that $y = 0$. Thus $TA$ is a closed operator.
\end{proof}

\begin{lem}
\label{lem:kauf_inv_closed}
\textsl{
Let $T$ be a closed operator on a Hilbert space $\sH$. 
\begin{itemize}
    \item[(i)] Then $$\ran(T^{\dagger}) = \dom(T) \cap \nul(T)^{\perp}.$$ Thus $T^{\dagger}$ gives a bijection between $\ran(T)$ and $\dom(T) \cap \nul(T)^{\perp}$. 
    \item[(ii)] Furthermore, $T^{\dagger}$ is closed and is the unique operator satisfying the following properties:
\begin{equation}
    T^\dagger T = (I - \nulN(T))|_{\dom(T)}, \; \; 
    TT^\dagger = I_{\sH} |_{\ran(T)},\; \;
    \dom(T^\dagger) = \ran(T).
    \label{eqn:char_of_Kaufman_inv}
\end{equation}
\end{itemize}
}
\end{lem}
\begin{proof} 
\begin{itemize}[leftmargin=0.6cm, itemsep=0.2cm]
\item[(i)] Since $T$ is closed, $\nul(T)$ is a closed subspace of $\sH$ (see \cite[Exercise 2.8.45]{kr-I}). Let $x$ be a vector in $\dom(T)$. Note that $\nulN(T)x \in \overline{\nul(T)} = \nul(T)\subseteq \dom(T)$. Thus $\big( I - \nulN(T) \big) x\in \dom(T) \cap \nul(T)^\perp$. We conclude that $\ran(T^\dagger) \subseteq \dom(T) \cap \nul(T)^\perp$. For the reverse inclusion, note that $\ran \big( I - \nulN(T) \big) = \nul (T)^\perp$. Hence if $x \in \dom(T)\cap \nul(T)^\perp$, then $x = \big( I - \nulN(T) \big) x \in \ran(T^\dagger)$.

\item[(ii)] Let $T_0$ be the restriction of $T$ to the subspace $\ran(T^\dagger)=\dom(T)\cap \nul(T)^\perp$. Note that $T_0$ is a bijection from the space $\ran(T^\dagger)$ onto $\dom(T^\dagger)=\ran(T)$. It is then clear from the definition of $T^\dagger$ that $T_0^{-1} = T^\dagger$. 

Since $T$ is closed and $\dom \left( \left( I - \nulN(T) \right)^{\dagger} \right) = \nul(T)^{\perp}$ is a closed subspace of $\sH$, by Remark \ref{rem:kauf_inv_1} and Lemma \ref{lem:unbdd_0}, we have that $T_0 = T(I - \nulN(T))^{\dagger}$ is a closed operator. Moreover, $\ran(T_0) = \ran(T)$. Since $\graph{T_0 ^{-1}}$ is obtained by flipping coordinates of $\graph{T_0} \subseteq \sH \oplus \sH$, we note that $\graph{T_0^{-1}}$ is closed, that is, $T_{0}^{-1}$ is a closed operator. In summary, $T^{\dagger}$ is closed whenever $T$ is closed. The equations in (\ref{eqn:char_of_Kaufman_inv}) are a matter of unwrapping the definitions, and it is straightforward to see that they completely determine $T^{\dagger}$. \qedhere
\end{itemize}
\end{proof}

\begin{remark}
\label{rem:kauf_inv_2}
Let $A\in \Bh$. For any unbounded operator $T$ on $\sH$ satisfying 
$TA = I - \nulN(A)$, note that $T$ is an extension of $A^{\dagger}$. Indeed,  $\{x \in \sH: Ax \in \dom(T) \} = \dom(TA) = \sH$, that is, $\dom(A^{\dagger}) = \ran(A) \subseteq \dom(T)$. Also, $T(Ax) = (I-\nulN(A))x = A^{\dagger}(Ax)$ for every $x\in \sH$; in other words, $A^{\dagger}$ and $T$ agree on $\ran(A) = \dom(A^{\dagger})$. Thus $A^\dagger\subseteq T$. This algebraic observation will be helpful in the proof of Theorem \ref{thm:uniqueness_phi_aff_2}.
\end{remark}

\begin{remark}
\label{rem:cl_bdd}
An operator on $\sH$ (not necessarily everywhere-defined on $\sH$) is closed and bounded if and only if it is the restriction of an everywhere-defined bounded operator on $\sH$, that is, an operator in $\sB(\sH)$, to a closed subspace of $\sH$. This follows from the well-known fact that a densely-defined closed operator is bounded if and only it is everywhere-defined. We denote the set of closed and bounded operators on $\sH$ by $\sC\sB(\sH)$. In view of Remark \ref{rem:kauf_inv_1}, every operator in $\sC\sB(\sH)$ is of the form $AE^{\dagger}$ for $A, E \in \sB(\sH)$ where $E$ is a projection.
\end{remark}

\begin{definition}[Adjoint]
\label{def:adj_not_dd_op}
For an unbounded operator $T$ acting on the Hilbert space $\sH$, $T : \dom(T) \to \sH$, we define an unbounded operator on $\sH$, $T^* : \dom(T^*) \to \sH$, called the {\it adjoint of $T$}, as follows. Its domain is given by:
    \[
        \dom(T^*) = \{ y \in \sH : \text{ the mapping } x \mapsto \ipdt{Tx}{y} \big( x \in \dom(T) \big)  \text{ is continuous } \},
    \]
By the Riesz representation theorem (see \cite[Theorem 2.3.1]{kr-I}) on the Hilbert space $\overline{\dom(T)}$, for every $y \in \dom(T^*)$ there is a \emph{unique} $z \in \overline{\dom(T)}$ such that $\ipdt{Tx}{y} = \ipdt{x}{z}$ for all $x \in \dom(T)$. We define $T^*y=z$. Clearly, the range of $T^*$ lies in $\overline{\dom(T)}$.
\end{definition}

\noindent {\bf \large  Caution.} In the literature, the adjoint is typically defined for a densely-defined operator. As seen above, we may view $T$ as a densely-defined operator from the Hilbert space $\overline{\dom(T)}$ to the Hilbert space $\sH$ and make perfectly good sense of its adjoint. It is crucial to remember that in our definition, the codomain of $T^*$ is $\sH$ and {\bf not} $\overline{\dom(T)}$, though we have that $\ran(T^*) \subseteq \overline{\dom(T)}$. The reader may verify that the standard definition of adjoint (see \cite[Definition 2.7.4]{kr-I}) very much depends on the codomain Hilbert space as any vector in $\ran(T)^{\perp}$ must lie in the domain of $T^*$. We hope that this clarification will be helpful in interpreting the meaning of $T^{**}$ as the adjoint of $T^*$, and in verifying that there is no sleight-of-hand involved when we invoke standard results from the literature about the adjoint in our proofs.

Using the same line of argument given in \cite[Remark 2.7.6]{kr-I}, note that for every unbounded operator $T$ on $\sH$, its adjoint $T^*$ is closed on $\sH$.

\begin{prop}[cf. {\cite[Theorem 2.7.8]{kr-I}}]
\label{prop:kr_I_278}
\textsl{
If $T$ is a densely-defined transformation from a Hilbert space $\sK$ to a Hilbert space $\sH$, then
\begin{itemize}
    \item[(i)] $T$ is pre-closed if and only if $\dom(T^*)$ is dense in $\sH$,
    \item[(ii)] if $T$ is pre-closed, $\overline{T}=T^{**}$.
\end{itemize}
}
\end{prop}

\begin{prop}
\label{prop:t_bar}
\textsl{
Let $T$ be a pre-closed operator (not necessarily densely-defined) acting on the Hilbert space $\sH$. Then $T^*$ is densely-defined and we have $\overline{T} = T^{**} E^{\dagger}$ where $E$ is the projection onto $\overline{\dom(T)}$.  
}
\end{prop}
\begin{proof}
By Proposition \ref{prop:kr_I_278}-(i) in the context of the Hilbert space $\sK = \overline{\dom(T)}$ and viewing $T$ as a pre-closed densely-defined operator from $\sK$ to $\sH$, we see that $T^*$ is densely-defined on $\sH$.

For pre-closed operators $T_1, T_2$, on Hilbert spaces $\sH_1, \sH_2$, respectively, it is straightforward to verify that $T_1 \oplus T_2$ is a pre-closed operator on $\sH_1 \oplus \sH_2$, and $\overline{T_1 \oplus T_2} = \overline{T_1} \oplus \overline{T_2}$.

Let $E$ denote the projection onto $\overline{\dom(T)}$. Clearly,
\[
\dom(TE) = \dom(T) \oplus \dom(T)^\perp, \text{ and } TE = T \oplus 0|_{\dom(T)^\perp}.
\]
Thus $TE$  is a pre-closed operator on the Hilbert space $\overline{\dom(T)} \oplus \dom(T)^{\perp}$ with closure $\overline{TE}= \overline{T}\oplus 0|_{\dom(T)^\perp}$. It follows that, $$\overline{T} = (\overline{T}\oplus 0|_{\dom(T)^\perp})E^\dagger = (\overline{TE})E^\dagger.$$ It may be directly verified from the definition of adjoint (Definition \ref{def:adj_not_dd_op}) that $(TE)^* = T^*$. Since $TE$ is densely-defined on $\sH$, using Proposition \ref{prop:kr_I_278} we see that $\overline{TE} = (TE)^{**} = T^{**}$. Hence,
\[
\overline{T} = (\overline{TE})E^\dagger = (TE)^{**} E^\dagger = T^{**} E^\dagger. \qedhere
\]
\end{proof}

\begin{prop}[Rank-nullity theorem]
\label{prop:rank_nullity}
\textsl{
For a densely-defined closed operator $T$ acting on $\sH$ we have:
    \[
        \raR(T)=I-\nulN(T^*), ~\nulN(T)=I-\raR(T^*).
    \]
}
\end{prop}
\begin{proof}
See \cite[Exercise 2.8.45]{kr-I}.
\end{proof}
\begin{definition}
\label{def:char_mat}
For an unbounded linear operator $T$ on $\sH$, the {\it characteristic projection} of $T$, denoted $\rchi(T)$, is defined to be the orthogonal projection from $\sH \oplus \sH$ onto the closure of the graph of $T$. We may view $\rchi(T)$ as a $2 \times 2$ matrix of operators in $\sB(\sH)$.
\end{definition}

\begin{remark}
Let $T$ be a closed operator acting on the Hilbert space $\sH$ with  characteristic projection given by, 
$$\rchi(T)=\begin{bmatrix}
    E_{11} & E_{12}\\
    E_{21} & E_{22}
\end{bmatrix} \in M_2 \big( \sB(\sH) \big),$$
where $E_{ij} \in \sB(\sH)$ for $1 \le i, j \le 2$. Since $T$ is a closed operator, $\ran \big( \rchi(T) \big) = \graph{T}$, so that $\dom(T)$ is the set of first coordinates of points in $\graph{T}$. Note that, 
$$\rchi(T)\begin{bmatrix}
    x\\y
\end{bmatrix} = \begin{bmatrix}
            E_{11} & E_{12} \\
            E_{21} & E_{22}
        \end{bmatrix}
        \begin{bmatrix}
            x \\ y
        \end{bmatrix}=
        \begin{bmatrix}
            E_{11}x + E_{12}y\\
            E_{21}x + E_{22}y
        \end{bmatrix}, \text{ for }
\begin{bmatrix}
x\\
y
\end{bmatrix} \in \sH \oplus \sH.$$ Thus,
    \begin{equation}
    \label{eqn:dom_from_char_mat}
    \dom(T) = \ran(E_{11}) + \ran(E_{12}).
    \end{equation}
\end{remark}

\begin{definition}[MvN-Affiliated Operators]
\label{def:mvn_aff}
Let $\kM$ be a von Neumann algebra acting on the Hilbert space $\sH$. A closed operator $T$ on $\sH$ is said to be MvN-\emph{affiliated} with $\kM$ if $U^*TU = T$ holds for all unitary $U$ in $\kM'$, the commutant of $\kM$. As mentioned earlier, the set of all MvN-affiliated operators will be denoted by $\afm^{\text{MvN}}$.
\end{definition}

\begin{remark}
Let $T$ be an closed operator on $\sH$ and $U$ be a unitary such that $UT\subseteq TU$. Then parsing the definitions, one sees that $U(\dom(T))\subseteq \dom(T)$. It is left as an exercise to the reader to verify that the following three statements are equivalent, providing alternate ways of viewing the notion of MvN-affiliation.
    \begin{enumerate}
        \item $TU=UT$ for all $U$ unitary in $\kM'$.
        \item  $U^*TU = T$ for all $U$ unitary in $\kM'$, that is, $T\in \afm^{\text{MvN}}$.
        \item $UT\subseteq TU$ for all $U$ unitary in $\kM'$.
    \end{enumerate}
\end{remark}

\begin{lem}
\label{lem:unbdd}
\textsl{
Let $\mathscr{M}$ be a von Neumann algebra acting on the Hilbert space $\sH$. For $T \in \afm^{\text{MvN}}$ and $A\in \kM$, we have $T^\dagger, TA$ lie in $\afm^{\text{MvN}}$. Moreover, if $\ran(A) \subseteq \dom(T)$, then $TA\in \kM$.
}
\end{lem}
\begin{proof}
First note that by Lemma \ref{lem:kauf_inv_closed}-(ii), $T^{\dagger}$ is closed and by Lemma \ref{lem:unbdd_0}, $TA$ is closed.

Let $U$ be any unitary in $\kM'$, the commutant of $\kM$. Since $UT = TU$ and $U^*T = TU^*$, it is clear that $U \nul (T) = \nul (T)$ so that $U (\nulN(T)) = \nulN(T) U$.  For every $x\in \sH$, we have,
\[
(UT^\dagger)Tx = U \big( I-\nulN(T) \big)x = \big( I - \nulN(T) \big)Ux= T^\dagger TUx = (T^\dagger U)Tx.
\]
Since $\dom(T^{\dagger}) = \ran(T)$, we conclude that $UT^{\dagger} = T^{\dagger}U$. Hence, $T^{\dagger} \in \afm^{\text{MvN}}$.

Similarly, since $UT = TU$ and $UA = AU$, we have $U(TA) = TUA = (TA)U$. Thus $TA \in \afm^{\text{MvN}}$.

If, in addition, $\ran (A) \subseteq \dom(T)$, then $\dom(TA) = \sH$ so that $TA$ is a bounded operator in $(\kM ')' = \kM$.
\end{proof}

The following lemma is a workhorse in this article in our pursuit of functoriality.
\begin{lem}
\label{lem:morph_prop}
\textsl{
Let $(\mathscr{M}; \sH), (\mathscr{N}; \sK)$ be represented von Neumann algebras and $\Phi : \mathscr{M} \to \mathscr{N}$ be a unital normal $*$-homomorphism.  For $A \in \kM$, we have 
\begin{align*}
    \Phi \big( \raR(A) \big) &= \raR \big( \Phi(A) \big),\\
    \Phi \big( \nulN(A) \big) &= \nulN \big( \Phi(A) \big).
\end{align*}
}
\end{lem}

\begin{proof}
Let $H$ be a positive contraction in $\mathscr{M}$, that is, $\| H \| \le 1$. Note that $\Phi(H)$ is a positive contraction in $\mathscr{N}$. From \cite[Lemma 5.1.5]{kr-I}, we note that $H^{\frac{1}{n}} \uparrow \raR(H)$ and $\Phi(H)^{\frac{1}{n}} \uparrow \raR \big( \Phi(H) \big)$. Using normality of $\Phi$, we have
$$\Phi \big( \raR(H) \big) = \Phi \big( \lim_{n \to \infty} H^{\frac{1}{n}} \big) = \lim_{n \to \infty} \Phi(H^{\frac{1}{n}}) = \lim_{n \to \infty} \Phi(H)^{\frac{1}{n}} = \raR \big( \Phi(H) \big).$$

If $A=0$, then the assertion is obvious. Thus we may assume that $A \in \kM$ is not the zero operator.
From Corollary \ref{cor:ran_aa_star} (or any proof of the polar decomposition theorem), for a bounded operator $A$ on a Hilbert space, we have $\ran(A) = \ran(|A^*|)$. Thus for $H := \frac{1}{\|A\|} |A^*|$ and $\Phi(H) = \frac{1}{\|A\|} |\Phi(A)^*|$, we get $$\raR(A) = \raR(H), \raR \big( \Phi(A) \big) = \raR \big( \Phi(H) \big).$$  Since $H$ is a positive contraction, by the discussion in the preceding paragraph, we conclude that $$\raR \big( \Phi(A) \big) = \raR \big( \Phi(H) \big) = \Phi \big( \raR(H) \big) = \Phi \big( \raR(A) \big).$$
Using the rank-nullity theorem (see Proposition \ref{prop:rank_nullity}), we have, 
\[
\Phi \big( \nulN(A) \big) = \nulN \big( \Phi(A) \big). \qedhere
\]
\end{proof}

\subsection{Near-semirings}
\label{subsec:near_semiring}
The concept of near-semirings was introduced in \cite[see \S 2]{rootselaar_1963} under the name `Fasthalbringe'. Unlike \emph{near-rings} for which a rich theory has already been developed, there does not appear to be a significant body of research in the direction of near-semirings. Even the definition of a near-semiring is not consistent in the literature, although the various definitions are closely related; They are sometimes referred to as {\it seminearrrings}. The definition we will use in this article is the one used in \cite[\S Introduction]{samman_1997}, \cite[pg. 714]{samman_2009}, and \cite[Definition 1.1]{kumar_krishna_2014}; We state it below.
 
\begin{definition}[Near-semiring]
\label{def:near-semiring}
    An algebraic structure $(S, +, \circ)$ with binary operations $(+, \circ)$ is said to be a \emph{right near-semiring} if it satisfies the following axioms: 
    \begin{enumerate}
        \item $(S, +)$ and $(S, \circ)$ are semigroups, and
        \item $(a+b) \circ c = a\circ c + b\circ c$ for all $a, b, c \in S$
    \end{enumerate}
We use the term {\it near-semiring} to mean {\it right near-semiring}. If instead of axiom (2), left-distributivity is assumed in the definition, then it is called a {\it left near-semiring}.

Let $S_1, S_2$ be near-semirings. A mapping $\varphi : S_1 \to S_2$ is said to be a {\it near-semiring homomorphism} if $\varphi(a+b) = \varphi(a) + \varphi(b)$ and $\varphi(a \circ b) = \varphi(a) \circ \varphi(b)$ for all $a, b \in S_1$.
\end{definition}

\begin{example}
    The set $\N \cup \{ 0 \}$ of all non-negative integers with the usual addition and multiplication is a near-semiring.
\end{example}

\begin{example}
\label{ex:m_gamma}
    Let $(\Gamma, +)$ is a semigroup and $M(\Gamma)$ denotes the set of all mappings from $\Gamma$ into $\Gamma$. For two mappings $f, g \in M(\Gamma)$ define the addition and multiplication as follows:
    \begin{align*}
        (f + g)(x) &:= f(x) + g(x) \text{, and }\\
        (f\circ g)(x) &:= f(g(x)),
    \end{align*}
    for all $x \in \Gamma$. Then $(M(\Gamma), + , \circ)$ is a near-semiring.
\end{example}

Example \ref{ex:m_gamma} does {\bf not} satisfy left-distributivity in general. In \S \ref{subsec:alg_str_m_aff}, we'll see that for a von Neumann algebra $\kM$, the set of affilitated operators (see Definition \ref{defn:M_aff}) is a near-semiring with respect to sum and product of (unbounded) operators, and left-distributivity is {\bf not} satisfied in general (see Remark \ref{rem:distributivity_operators}). 

As mentioned above, a literature review of the topic of near-semirings may lead to some confusion owing to a lack of consensus on the core definitions. Keeping this in mind, readers who wish to explore this topic further may refer to \cite{hoorn_willy_root_1967}, \cite{krishna_thesis_2005}, \cite{kumar_krishna_2014}, to get a sense of the motivation for the various defining sets of axioms for near-semirings.

\section{The Lattice of Affiliated Subspaces}%
\label{sec:aff_subspaces}%

In this section, we present the concept of \emph{affiliated subspaces} for a represented von Neumann algebra. Our key observation about the set of affiliated subspaces is that it forms a lattice under vector sum as {\it join} and set intersection as {\it meet}, and transforms functorially under unital normal $*$-homomorphisms between represented von Neumann algebras.  A crucial role in our discussion is played by the Douglas factorization lemma in the context of von Neumann algebras, which is why we begin with its study.

\subsection{The Douglas Factorization Lemma and $ \affs{\kM}$}

\begin{thm}[{cf.\ \cite{douglas_lemma_paper}, \cite[Theorem 2.1]{nayak_douglas_vNa}}]
\label{thm:douglas_vNa}
\textsl{
Let $\mathscr{M}$ be a von Neumann algebra acting on the Hilbert space $\sH$. Let $A, B \in \mathscr{M}$. Then the following conditions are equivalent:
\begin{itemize}
    \item[(i)] $\ran (A) \subseteq \ran (B)$;
    \item[(ii)] $B^{\dagger}A$ is a bounded operator and lies in $\kM$;
    \item[(iii)] $A=BX$ for some $X \in \mathscr{M}$;
    \item[(iv)] $AA^* \le \lambda ^2 BB^*$ for some $\lambda > 0$.
\end{itemize}
}
\end{thm}
\begin{proof}
\noindent {\large (i) $\Rightarrow$ (ii)}. Since $\ran(A)\subseteq \ran(B) \subseteq \dom(B^\dagger)$ and $B^{\dagger} \in \afm^{\text{MvN}}$, by Lemma \ref{lem:unbdd}, $X = B^{\dagger}A$ is in $\mathscr{M}$. 

\vskip 0.1in
\noindent {\large (ii) $\Rightarrow$ (iii).} Let $X := B^{\dagger}A \in \kM$. Since $BB^{\dagger} : \ran(B) \to \ran(B)$ is the identity mapping and $\ran (A) \subseteq \ran(B)$, we conclude that $A = (BB^{\dagger})A = BX$.

\vskip 0.1in
\noindent {\large (iii) $\Rightarrow$ (i)}. Straightforward to see.

\vskip 0.1in
\noindent {\large (iii) $\Rightarrow$ (iv)}. Since $AA^* = BXX^*B^* \le \|X\|^2BB^*$, we may choose $\lambda > \|X\|$.

\vskip 0.1in

\noindent {\large (iv) $\Rightarrow$ (iii)}. For all $x \in \sH$, we have $\|A^*x\| \le \lambda \|B^*x\|$, so that whenever $B^*x = 0$, we have $A^*x = 0$. Thus the mapping $D : \ran(B^*) \to \ran(A^*)$ given by $D(B^*x) = A^*x$ is a well-defined bounded linear mapping. We extend the definition of $D$ to $\overline{\ran(B^*)}$ by continuity. 

Let $U$ be a unitary operator in $\mathscr{M}'$ so that $U$ commutes with $A^*, B^*$. For every $x \in \sH$, we have $$DUB^*x = DB^*Ux = A^*Ux = UA^*x = UDB^*x,$$ so that $D \in \afm^{\text{MvN}}$.  By Lemma \ref{lem:unbdd}, $D\,\raR(B^*) \in \mathscr{M}$. For $X  = \big( D \raR(B^*) \big)^* \in \kM$, we have $A^* = X^*B^*$ so that $A = BX$.
\end{proof}

\begin{cor}
\label{cor:ran_aa_star}
\textsl{
Let $\sH$ be a Hilbert space and $A \in \sB (\sH)$. Then $\ran(A)=\ran(|A^*|)$.
}
\end{cor}
\begin{proof}
Consider the self-adjoint operator $B=|A^*|$. Clearly, $AA^* = BB^*$. The assertion follows from the equivalence ``$(i) \iff (iv)$'' in the statement of Theorem \ref{thm:douglas_vNa}.
\end{proof}

\begin{prop}
    \label{prop:douglas_soln}
\textsl{
Let $\kM$ be a von Neumann algebra acting on the Hilbert space $\sH$ and let $A, B\in \kM$ satisfy $\ran(A) \subseteq \ran(B)$. Then $B^{\dagger}A$ is the unique operator $X$ in $\kM$ satisfying $BX = A$ and $\raR(X) \le \raR(B^*)$.
}
\end{prop}
\begin{proof}
Note that the condition, $\raR(X) \le \raR(B^*)$, is equivalent to the condition, $\ran(X) \subseteq \overline{\ran(B^*)}$. By taking the orthogonal complement of both sides, we see that $$\nul (B) \subseteq \nul (X^*).$$

Let $X_1, X_2$ be two operators satisfying the given conditions. Then we get $X_1^*B^* = A^* = X_2^* B^*$. Thus $X_1 ^*= X_2^*$ on $\ran(B^*)$.
If $x\in \ran(B^*)^\perp=\nul(B) \subseteq \nul (X_1^*) \cap \nul (X_2^*)$, then $X_1^*x = 0 = X_2^*x$. Thus $X_1^* = X_2^*$, or equivalently, $X_1 = X_2$. This proves the uniqueness of such $X$ (if one exists).

As seen in the proof of Theorem \ref{thm:douglas_vNa}, clearly $B(B^{\dagger}A) = A$, and from the definition of the Kaufman-inverse, we have
\[
\ran (B^{\dagger}A) \subseteq \ran(B^{\dagger}) \subseteq \nul (B)^{\perp} = \overline{\ran(B^*)}. \qedhere
\]
\end{proof}

\begin{lem}
\label{lem:phi_pres_douglas}
\textsl{
Let $(\kM; \sH), (\kN; \sK)$ be represented von Neumann algebras and $\Phi:\kM\to \kN$ be a unital normal $*$-homomorphism. Let $A, B \in \kM$ with $\ran(A) \subseteq \ran(B)$. Then we have:
\begin{itemize}
    \item[(i)] $\ran \big( \Phi(A) \big) \subseteq \ran \big( \Phi(B) \big)$;
    \item[(ii)] $\Phi \left( B^{\dagger}A \right) = \Phi(B)^{\dagger} \Phi(A)$ in $\kN$.
\end{itemize}
}
\end{lem}
\begin{proof}
Let $X:=B^{\dagger}A$. As noted in Proposition \ref{prop:douglas_soln}, $X \in \kM$, $BX=A$ and $\raR(X) \le \raR(B^*)$. 
\vskip 0.1in

\begin{itemize}[leftmargin=0.7cm, itemsep=0.2cm]
    \item[(i)] Since $\Phi(B) \Phi(X) = \Phi(A)$, we see that $\ran \big( \Phi(A) \big) \subseteq \ran \big( \Phi(B) \big)$. 
    \item[(ii)] Since $\raR \big( \Phi(X) \big) \le \raR \big( \Phi(B)^* \big)$ (by Lemma \ref{lem:morph_prop}), Proposition \ref{prop:douglas_soln} tells us that,
    \[
    \Phi(X) = \Phi(B)^{\dagger}\Phi(A).   \qedhere 
    \]
\end{itemize}
\end{proof}

\begin{thm}[cf. {\cite[Theorem 2.2, Corollary 2]{fillmore-williams}}]
\label{thm:boxplusdot}
\textsl{
Let $\kM$ be a von Neumann algebra acting on the Hilbert space $\sH$. Let $A, B$ be operators in $\kM$ and $\fT_{A, B}$ denote the operator in $M_2(\kM)$ as defined in \S \ref{sec:prelim}. Then we have the following.
\begin{itemize}[leftmargin=0.8cm, itemsep=0.2cm]
    \item[(i)] The range of the operator $|\fT_{A,B}^*|_{\raisemath{-2pt}{11}} \in \kM$ is $\ran(A) + \ran(B)$.
    \item[(ii)] The range of the operator $\sqrt{A\;\nulN (\fT_{A, B})_{\raisemath{-2pt}{11}}A^*} \in \kM$ is $\ran(A) \cap \ran(B)$.
    \item[(iii)] The range of the operator $\sqrt{\nulN (\fT_{A, B})_{\raisemath{-2pt}{11}}} \in \kM$ is $A^{-1}\Big( \ran(B) \Big)$. 
\end{itemize}
}
\end{thm}
\begin{proof}
\begin{itemize}[leftmargin=0.8cm, itemsep=0.2cm]
    \item[(i)] By applying $\fT_{A, B}$ to vectors in $\sH \oplus \sH$, it is straightforward to see that $$\ran(\fT_{A, B}) = \big( \ran(A) + \ran(B) \big) \oplus \{0 \}_{\sH}.$$ Note that, 
\[
|\fT_{A,B}^*| = \begin{bmatrix}
\sqrt{AA^*+BB^*} & 0\\
0 & 0
\end{bmatrix}.
\]
From Corollary \ref{cor:ran_aa_star} in the context of the von Neumann algebra $M_2(\kM)$, we have 
\begin{align*}
\big( \ran(A) + \ran(B) \big) \oplus \{ 0 \}_{\sH} &= \ran(\fT_{A, B}) = \ran(|\fT_{A, B}^*|) \\
&= \ran(\sqrt{AA^*+BB^*}) \oplus \{ 0 \}_{\sH}.
\end{align*}

    \item[$\frac{\text{(ii)}}{\text{(iii)}}$] Note that for $x, y \in \sH$, the vector $(x, y) \in \sH \oplus \sH$ is in $\nul (\fT_{A, B})$ if and only if $Ax = By$. Thus $x$ is the first coordinate of a vector in $\nul (\fT_{A, B})$ if and only if $Ax \in \ran(B)$ or equivalently, $Ax \in \ran(A) \cap \ran(B)$. Thus,

\begin{align*}
\ran \left( \begin{bmatrix}
A & 0\\
0 & 0
\end{bmatrix} \nulN(\fT_{A, B}) \right) &= \left( \ran(A) \cap \ran(B) \right) \oplus \{ 0 \}_{\sH}\\
\ran \left( \begin{bmatrix}
I & 0\\
0 & 0
\end{bmatrix} \nulN(\fT_{A, B}) \right) &= A^{-1} \big( \ran(B) \big) \oplus \{ 0 \}_{\sH}.
\end{align*}

Since 
\begin{align*}
\begin{bmatrix}
A & 0\\
0 & 0
\end{bmatrix} \nulN(\fT_{A, B}) 
\begin{bmatrix}
A^* & 0\\
0 & 0
\end{bmatrix} &= \begin{bmatrix}
A\nulN(\fT_{A, B})_{11} A^* & 0\\
0 & 0
\end{bmatrix}\\
\begin{bmatrix}
I & 0\\
0 & 0
\end{bmatrix} \nulN(\fT_{A, B})
\begin{bmatrix}
I & 0\\
0 & 0
\end{bmatrix} &= \begin{bmatrix}
\nulN(\fT_{A, B})_{11} & 0\\
0 & 0
\end{bmatrix},
\end{align*}
using Corollary \ref{cor:ran_aa_star} in the context of $M_2(\kM)$, we conclude that,
\begin{align*}
\Big( \ran(A) \cap \ran(B) \Big) \oplus \{ 0 \}_{\sH} &= \ran \left( \sqrt{A\;\nulN (\fT_{A, B})_{\raisemath{-2pt}{11}}A^*} \right) \oplus \{ 0 \}_{\sH},\\
A^{-1}\Big(  \ran(B) \Big) \oplus \{ 0 \}_{\sH} &= \ran \left( \sqrt{\;\nulN (\fT_{A, B})_{\raisemath{-2pt}{11}}} \right) \oplus \{ 0 \}_{\sH}.   \qedhere
\end{align*}
\end{itemize}

\end{proof}

\begin{remark}
\label{rem:aff_lat}
Let $\kM$ be a von Neumann algebra acting on the Hilbert space $\sH$ and $A, B \in \kM$. In the context of Theorem \ref{thm:boxplusdot}, we denote the operator in parts (i), (ii), (iii), respectively, by $A \boxplus B$, $A \boxdot B$, $A \invran B$, respectively. Thus the ranges of the operators $A \boxplus B, A \boxdot B$ and $A \invran B$ in $\kM$  are given by:
\begin{itemize}
    \item[(i)] $\ran(A \boxplus B) = \ran(A) + \ran(B)$;
    \item[(ii)] $\ran(A \boxdot B) = \ran(A) \cap \ran(B)$;
    \item[(iii)] $\ran(A \invran B) = A^{-1} \left( \ran(B) \right)$.
\end{itemize}
In fact, from the proof of Theorem \ref{thm:boxplusdot}, we have an explicit formula for $A \boxplus B$, given by $\sqrt{AA^* + BB^*}$.
\end{remark}

\subsection{Functoriality of the Construction $\kM \mapsto \affs{\kM}$}

\begin{definition}
\label{def:aff_sub}
A subspace $\sV \subseteq \sH$ is said to be {\it affiliated} to $\kM$ if there is an operator in $\kM$ whose range is $\sV$. We denote the set of affiliated subspaces for $\kM$, or equivalently, the set of operator ranges for $\kM$, by $\affs{\kM}$.
\end{definition}

The rationale for the above nomenclature will be clear from Theorem \ref{thm:dom_ran} where it will be shown that $\affs{\kM}$ is precisely the set of domains of affiliated operators. It is well-known that the set of ranges of operators in $\sB(\sH)$ is a lattice under vector addition as join and set intersection as meet (see \cite{mackey1946domains}, \cite{dixmier1949}, \cite{fillmore-williams}). In Lemma \ref{lem:lat_of_ran} below, we show an analogous result in the context of von Neumann algebras.

\begin{lem}
	\label{lem:lat_of_ran}
\textsl{
Let $\mathscr{M}$ be a von Neumann algebra acting on the Hilbert space $\sH$.  
\begin{itemize}
    \item[(i)] For affiliated subspaces $\sV, \sW \in \affs{\kM}$, we have $\sV+ \sW, \sV \cap \sW \in \affs{\kM}$; in other words, $\text{Aff}_{s}(\mathscr{M})$ is a lattice under vector addition ($+$) and set intersection ($\cap$).
    \item[(ii)] For an affiliated subspace $\sV \in \affs{\kM}$, we have $\overline{\sV}, \sV + \sV^{\perp}$ are also affiliated subspaces in $ \affs{\kM}$. 
\end{itemize} 
}
\end{lem}
\begin{proof}
\begin{itemize}[leftmargin=0.7cm, itemsep=0.2cm]
    \item[(i)] Let $A, B  \in \kM$ be operators such that $\sV = \ran(A), \sW = \ran(B)$. From Remark \ref{rem:aff_lat}, we have 
\begin{align*}
\sV + \sW &= \ran(A) + \ran(B) = \ran(A \boxplus B) \in \affs{\kM},\\
\sV \cap \sW &= \ran(A) \cap \ran(B) =  \ran(A \boxdot B) \in \affs{\kM}.
\end{align*}
    \item[(ii)] Since $\mathbf{R}(A) \in \kM$ and $\overline{\sV} = \ran \big( \mathbf{R}(A) \big)$, we have $\overline{\sV} = \affs{\kM}$. Let $E$ be the projection onto $\overline{\sV}$. Then $\ran(I-E) = \sV^{\perp}$ and $\sV + \sV^{\perp} = \ran(A) + \ran(I-E) = \ran \big( A \boxplus (I-E) \big)$ which lies in $\affs{\kM}$ by part (i).   \qedhere
\end{itemize}
\end{proof}

\begin{lem}
\label{lem:phi_pres_boxes}
\textsl{
Let $(\mathscr{M}; \sH), (\mathscr{N}; \sK)$ be represented von Neumann algebras and $\Phi : \mathscr{M} \to \mathscr{N}$ be a unital normal $*$-homomorphism. For $A, B\in \kM$, we have 
\begin{align*}
    \Phi(A \boxplus B) &= \Phi(A) \boxplus \Phi(B),\\
    \Phi(A \boxdot B) &= \Phi(A) \boxdot \Phi(B),\\
    \Phi(A\invran B) &= \Phi(A)\invran \Phi(B).
\end{align*}
 }
\end{lem}

\begin{proof}
Consider the unital normal $*$-homomorphism $\Phi_{(2)}:M_2(\kM) \to M_2(\kN)$. Recall the notation $\fT_{A', B'}$ from \S \ref{sec:prelim}. It is straightforward to see that $\Phi_{(2)}(\fT_{A, B}) = \fT_{\Phi(A), \Phi(B)}$. Using Lemma \ref{lem:morph_prop}, we get $\Phi_{(2)} \big( \nulN (\fT_{A, B}) \big) = \nulN \big( \fT_{\Phi(A), \Phi(B)} \big)$. The desired result follows from Theorem \ref{thm:boxplusdot} and the three equations below,
\begin{align*}
\Phi_{(2)}(\fT_{A, B}) &= \fT_{\Phi(A), \Phi(B)};\\
\Phi_{(2)} \left( 
\begin{bmatrix}
A & 0\\
0 & 0
\end{bmatrix} \nulN(\fT_{A, B}) \right) &= 
\begin{bmatrix}
\Phi(A) & 0 \\
0 & 0
\end{bmatrix} \nulN \Big( \fT_{\Phi(A), \Phi(B)} \Big);\\
\Phi_{(2)} \left( 
\begin{bmatrix}
I_{\sH} & 0\\
0 & 0
\end{bmatrix} \nulN(\fT_{A, B}) \right) &= 
\begin{bmatrix}
I_{\sK} & 0 \\
0 & 0
\end{bmatrix} \nulN \Big( \fT_{\Phi(A), \Phi(B)} \Big).  \qedhere
\end{align*}
\end{proof}

\begin{thm}
\label{thm:lat_morph}
\textsl{
Let $(\mathscr{M}; \sH), (\mathscr{N}; \sK)$ be represented von Neumann algebras and $\Phi : \mathscr{M} \to \mathscr{N}$ be a unital normal $*$-homomorphism.
\begin{itemize}
    \item[(i)] The mapping $\Phi_{\aff}^s : \affs{\kM} \to \affs{\kN}$ given by $\ran(A) \mapsto \ran \big( \Phi(A) \big)$ for $A \in \mathscr{M}$, is a well-defined lattice homomorphism;
    \item[(ii)] $\Phi_{\aff}^s(\{ 0 \}_{\sH}) = \{ 0 \}_{\sK}$,  and $\Phi_{\aff}^s(\sH) = \sK$;
    \item[(iii)] for all $\sV \in \affs{\kM}$, $\Phi_{\aff}^s\left(\overline{\sV}\right)=\overline{\Phi_{\aff}^s(\sV)}$;
    \item[(iv)] if $\sV$ dense in $\sH$, then $\Phi_{\aff}^s(\sV)$ is dense in $\sK$;
    \item[(v)] if $\sV, \sW \in \affs{\kM}$ are such that $\sV \subseteq \sW$, then $\Phi_{\aff}^s(\sV) \subseteq \Phi_{\aff}^s(\sW)$.
\end{itemize}
}
\end{thm}
\begin{proof}
\begin{itemize}[leftmargin=0.8cm, itemsep=0.2cm]
    \item[(i)] Let $A, B \in \kM$ such that $\ran(A) = \ran(B) \subseteq \sH$. By Theorem \ref{lem:phi_pres_douglas}-(i), we have $\ran(\Phi(A)) = \ran (\Phi(B)) \subseteq \sK$. Thus $\Phi_{\aff}^s$ is well-defined.

Let $\sV, \sW \in \affs{\kM}$ so that there are operators $A, B \in \mathscr{M}$ with $\sV = \ran(A), \sW= \ran(B)$. Using Lemma \ref{lem:lat_of_ran} and Remark \ref{rem:aff_lat}, we note that,\[
\sV+\sW = \ran(A) + \ran(B) = \ran (A \boxplus B).
\] Thus
\begin{align*}
    \Phi_{\aff}^s (\sV+\sW) 
    &= \Phi_{\aff}^s\left( \ran(A \boxplus B) \right) = \ran\left( \Phi(A \boxplus B) \right) \\
    &= \ran\left(\Phi(A) \boxplus \Phi(B)\right) \text{, by Lemma \ref{lem:phi_pres_boxes}}\\
    &= \ran(\Phi(A)) + \ran(\Phi(B)), \text{ by Remark \ref{rem:aff_lat}}\\
    &= \Phi_{\aff}^s(\sV) + \Phi_{\aff}^s(\sW).
\end{align*}
On the other hand using Lemma \ref{lem:lat_of_ran} and Theorem \ref{thm:boxplusdot}-(ii), we get,
\[
    \sV \cap \sW = \ran(A)\cap \ran(B)= \ran(A \boxdot B).
\]
Hence,
\begin{align*}
\Phi_{\aff}^s(\sV \cap \sW)
    &= \Phi_{\aff}^s\left( \ran(A \boxdot B) \right) = \ran\left( \Phi(A \boxdot B) \right) \\
    &= \ran\left(\Phi(A) \boxdot \Phi(B)\right) \text{, by Lemma \ref{lem:phi_pres_boxes}}\\
    &= \ran \left( \Phi(A) \right) \cap \ran \left( \Phi(B) \right), \text{ by Remark \ref{rem:aff_lat}}\\
    &= \Phi_{\aff}^s(\sV) \cap \Phi_{\aff}^s(\sW).
\end{align*}

    \item[(ii)] Follows immediately from $\Phi(0_{\sH}) = 0_{\sK}$ and $\Phi(I_{\sH}) = I_{\sK}$, respectively.

    \item[(iii)] Let $A\in \kM$ be such that $\sV = \ran(A)$. Then we have,
\begin{align*}
\Phi_{\aff}^s(\overline{\sV}) &= \Phi_{\aff}^s\left( \overline{\ran(A)} \right) = \Phi_{\aff}^s\big( \ran ~ \raR(A) \big)\\
&= \ran\big( \Phi\left( \raR(A) \right) \big) = \ran \big( \raR\left(\Phi(A)\right) \big)\\
&= \overline{\ran \Phi(A)} = \overline{\Phi_{\aff}^s(\sV)}.
\end{align*}

    \item[(iv)] If $\sV$ is a dense affiliated subspace of $\sH$, then  $\overline{\Phi_{\aff}^s(\sV)} = \Phi_{\aff}^s(\overline{\sV}) = \Phi_{\aff}^s(\sH) = \sK$.

    \item[(v)] Let $\sV = \ran(A)$ and $\sW = \ran(B)$ for $A, B \in\kM$. Then by Lemma \ref{lem:phi_pres_douglas}-(i) we get, $\Phi_{\aff}^s(\sV) = \ran(\Phi(A)) \subseteq \ran(\Phi(B)) = \Phi_{\aff}^s(\sW) $.  \qedhere
\end{itemize}
\end{proof}

\begin{remark}
	Thus $\mathscr{M} \mapsto \affs{\kM}, \Phi \mapsto \Phi_{\aff}^s$ gives a functor from the category of (represented) von Neumann algebras to the category of lattices. In particular, we observe that the lattice of operator ranges is intrinsically associated with the von Neumann algebra independent of its representation on Hilbert space.
\end{remark}%

\section{Affiliated Operators}
\label{sec:aff_operators}

Let $\sH$ be a Hilbert space and $A, B \in \sB(\sH)$. There is a unique (unbounded) linear operator $T : \ran(B) \to \sH$ such that $TB = A$, if and only if $\nul (B) \subseteq \nul (A)$. Clearly, for every $x \in \sH$, the mapping $T$ must send $Bx$ to $Ax$, and such a mapping makes sense if and only if $\nul (B) \subseteq \nul (A)$. If $\nul (B) \subseteq \nul (A)$, we use the notation $A/B$ to denote the operator $T$ and call it the {\it quotient} of $A$ and $B$, or a {\it quotient representation} of the operator $T$; note that $A/B = A B^{\dagger}$.

We may draw the analogy that for integers $a, b$, there is a unique rational number $a/b$ such that $(a/b)b=a$ if and only if $b \ne 0$. The nullspace condition assists in the `cancellation' of the singular parts of $A$ and $B$, thereby making the expression $A/B$ sensible. If the nullspace condition is not met, the expression $A/B$ should be regarded as invalid syntax, much like how $a/0$ is deemed invalid syntax. For instance, the notation $I/B$ makes sense only if $\nulN(B) = 0$. Throughout this section, $\kM$ denotes a von Neumann algebra acting on the Hilbert space $\sH$.

\begin{definition}
\label{defn:M_aff}
In the above notation, we define the set of \emph{affiliated operators} for the von Neumann algebra $\kM$ to be the set of quotients of operators in $\kM$, that is,
\[
\afm := \{ A/B: A, B \in \kM \textrm{ with } \nul (B) \subseteq \nul (A) \}.
\]
\end{definition}

Our main goal here is to establish the basic algebraic structure of the set of affiliated operators and show that it transforms functorially under unital normal $*$-homomorphisms between represented von Neumann algebras. To this end, it becomes essential to gain a thorough understanding of the quotient representation of affiliated operators. 

\subsection{Equi-range Operators and Uniqueness of Quotient Representation}
Note that a right-invertible operator $C \in \sB(\sH)$ must be surjective. Thus, for $A, B \in \sB(\sH)$ with $\nul (B) \subseteq \nul (A)$, it is straightforward to see that $A/B = AC/BC$. The primary objective of this subsection is to obtain a converse result for quotients of operators in a represented von Neumann algebra, thus showing that the quotient representation of an affiliated operator is essentially unique (see Theorem \ref{thm:uniqueness_of_quotients}).

\begin{prop}
\label{prop:ran_equal}
\textsl{
Let $\mathscr{M}$ be a von Neumann algebra acting on the Hilbert space $\sH$. Let $A, B$ be two operators in $\kM$. Then $\ran(A) = \ran(B)$ if and only if there is a right-invertible element $C \in \kM$ and projections $P, Q$ in the centre of $\kM$ with $P+Q = I$ such that $AP=BCP$ and $BQ = ACQ$.
}
\end{prop}
\begin{proof}
\noindent {\large ($\Longrightarrow$)}. Since $\ran(A) = \ran(B)$, the operators $B^{\dagger}A, A^{\dagger} B$ are bounded operators in $\mathscr{M}$. Note that 
\begin{align*}
(B^{\dagger}A)(A^{\dagger}B) &= B^{\dagger}B =  I- \nulN(B),\\
(A^{\dagger}B)(B^{\dagger}A) &= A^{\dagger}A = I - \nulN(A).
\end{align*}
To begin with, let us assume that the projections $\nulN(A)$ and $\nulN(B)$ are comparable (say, $\nulN(B) \precsim \nulN(A)$); for instance, this will be the case when $\kM$ is a factor.

Let $V$ be a partial isometry in $\kM$ such that $VV^* = \nulN(B)$ and $V^*V \le \nulN(A)$, so that $A(V^*V) = 0$ and $B(VV^*) = 0$. Keeping in mind the basic identities for partial isometries, $VV^*V = V, V^*VV^* = V^*$, we note that
\begin{align*}
AV^* &= A(V^*V)V^* = 0,\\  
VA^{\dagger} &= V(V^*V) A^{\dagger} = 0,\\
BV &= B(VV^*)V = 0.
\end{align*}
Consider the operators $C := B^{\dagger}A + V, D := A^{\dagger}B + V^*$ in $\mathscr{M}$.  Using the above equalities, we have
\begin{align*}
    CD &= B^{\dagger}AA^{\dagger}B + B^{\dagger}AV^* + V A^{\dagger}B + VV^*\\
    &= \left( I - \nulN(B) \right) + 0 + 0 + \nulN(B)\\
    &= I, \text{ and }\\
BC &= BB^{\dagger}A + BV = A.
\end{align*}
Thus $C$ is a right-invertible element in $\kM$ with $A = BC$. In the case where $\nulN(B) \precsim \nulN(A)$, by a symmetric argument, we have a right-invertible element $C \in \kM$ such that $B = AC$. 

Next we proceed to the general case where $\nulN(A)$ and $\nulN(B)$ may not be comparable. By the comparison theorem, \cite[Theorem 6.2.7]{kr-II}, there are central projections $P, Q$ with $P+Q=I$ such that $\nulN(A)P \precsim \nulN(B)P$ and $\nulN(B)Q \prec \nulN(A)Q$. Consider the von Neumann algebra $\kM P$ acting on the Hilbert space $P(\sH)$, and view the operators $AP, BP$ as acting on $P(\sH)$. In the context of $\kM P, AP, BP$, the preceding discussion tells us that there is an operator $C_P \in \kM P$ which is right-invertible relative to $\kM P$ such that $AP = (BP)C_P$. Similarly there is a right-invertible operator $C_Q \in \kM Q$ such that $BQ = (AQ) C_Q$. The operator $C := C_P \oplus C_Q$ is a right-invertible operator in $\kM$ satisfying the desired properties.
\vskip 0.08in

\noindent {\large  ($\Longleftarrow$).} Since $AP = BCP = BPC$, we have $\ran(AP) \subseteq \ran(BP)$. Let $D$ be the right-inverse of $C$. Note that $(AP)D = ADP = BCDP = BP$ whence $\ran(BP) \subseteq \ran(AP)$. Thus $\ran(AP) = \ran(BP)$ and similarly, $\ran(AQ) = \ran(BQ)$. By Remark \ref{rem:aff_lat}, $$\ran (AP) + \ran(AQ) = \ran (\sqrt{AA^*P+ AA^*Q}) = \ran \left( \sqrt{AA^*} \right) = \ran(A),$$ and similarly, $\ran(BP) + \ran(BQ) = \ran(B)$. We conclude that $\ran(A) = \ran(B)$.
\end{proof}

\begin{cor}
\textsl{
Let $\mathscr{M}$ be a von Neumann algebra acting on the Hilbert space $\sH$, and $A, B$ be two operators in $\kM$. 
\begin{itemize}
    \item[(i)] If $\kM$ is factor, then $\ran(A) = \ran(B)$ if and only if there is a right-invertible element $C \in \mathscr{M}$ such that either $A = BC$ or $B=AC$. 
    \item[(ii)] If $\kM$ is a finite von Neumann algebra, $\ran(A) = \ran(B)$ if and only if there is an invertible element $C \in \mathscr{M}$ such that $A = BC$.
\end{itemize}
}
\end{cor}
\begin{proof}
Since the only central projections in a factor are $0$ and $I$, part (i) follows from Proposition \ref{prop:ran_equal}. Part (ii) follows from the observation that every right-invertible element in a finite von Neumann algebra is invertible, and the fact that $A=BC$ if and only if $A=BC^{-1}$.
\end{proof}

In the theorem below, we establish that the quotient representation for affiliated operators is essentially unique.

\begin{thm}[Uniqueness of quotient representation]
\label{thm:uniqueness_of_quotients}
\textsl{
Let $\mathscr{M}$ be a von Neumann algebra acting on the Hilbert space $\sH$. For $i = 1, 2$, let $A_i, B_i$ be bounded operators in $\kM$ with $\nul (B_i) \subseteq \nul (A_i)$. Then $A_1/B_1 = A_2/B_2$ if and only if there is a right-invertible element $C \in \kM$ and projections $P, Q$ in the centre of $\kM$ with $P+Q=I$ such that 
\begin{align*}
    A_1 P &= A_2 C P, ~ B_1 P = B_2 C P, \text{ and } \\
    A_2 Q &= A_1 C Q, ~ B_2 Q = B_1 C Q.
\end{align*}
}
\end{thm}

\begin{proof}
($\Longrightarrow$) For $S = A_1/B_1 = A_2/B_2$, we have, 
\begin{equation}
    \label{eqn:t_b_i_a_i}
    SB_i x = A_i x. \; (\forall x \in \sH \text{ and } i = 1, 2).
\end{equation}

Since $\ran(B_1) = \ran(B_2) (=\dom(S))$, using Proposition \ref{prop:ran_equal}, we observe that there exists a right invertible element $C\in \kN$ and two projections $P, Q$ in the centre of $\kN$ with $P+Q=I$ such that 
\[
B_1 P = B_2 CP \text{ and } B_2 Q = B_1 C Q.
\]

Using equation (\ref{eqn:t_b_i_a_i}), for every $x \in \sH$, we have $SB_1(Px) = A_1(Px)$ and $SB_2(CPx)=A_2CPx$. Since $B_1 P = B_2 CP$, we note that $A_1Px = SB_1 Px = SB_2CPx = A_2 CPx$ for every $x \in \sH$, whence $A_1 P = A_2 CP$. In a similar way from $B_2 Q = B_1 C Q$ we get $A_2 Q = A_1 C Q$. This completes the forward implication of the theorem.
\vskip 0.05in

\noindent ($\Longleftarrow$) Assume that there are operators $C, P, Q$ in $\kM$ as described in the assertion. First we observe that the existence of such $C, P, Q$ necessarily implies (by Proposition \ref{prop:ran_equal}) that, $$\dom(A_1/B_1)= \ran(B_1) = \ran(B_2) = \dom(A_2/B_2).$$
Now note that for any $x\in \sH$,
\[
(A_1/B_1)(B_1Px) = A_1 Px=A_2 CPx = (A_2/B_2) (B_2 CPx) = (A_2/B_2)(B_1 P x).
\]
Hence $A_1/B_1= A_2/B_2$ on $\ran(B_1 P)$. A similar calculation shows that $(A_2/B_2) (B_2Qx)= (A_1/B_1) (B_2 Q x)$ for all $x\in \sH$, that is, $A_1/B_1 = A_2/B_2$ on $\ran(B_2 Q)$. As $C$ is surjective (being right-invertible), we have $\ran (B_2 Q) = \ran(B_1 CQ) = \ran(B_1QC) = \ran(B_1Q)$. Note that, $$\ran(B_1P) + \ran(B_1Q) = \ran\left( \sqrt{B_1P B_1^* + B_1 Q B_1^*} \right)=\ran \left( \sqrt{B_1B_1^*} \right) = \ran(B_1).$$ Therefore we get,
\[
A_1/B_1= A_2/B_2 \text{ on }\ran(B_1).
\]
Hence, $A_1/B_1 = A_2/B_2$.
\end{proof}

\begin{cor}
\textsl{
Let $\kM$ be a factor acting on $\sH$ and $A_1, B_1, A_2, B_2 \in \kM$. If $A_1/B_1 = A_2/B_2$, then there is a right-invertible element $C \in \kM$ such that, 
\begin{align*}
    \text{either, }A_1 = A_2 C, ~ B_1 = B_2 C,\\
    \text{or, } A_2 = A_1 C,~ B_2 = B_1 C.
\end{align*}
}
\end{cor}
\begin{proof}
Since a factor only has trivial central projections, the result immediately follows from Theorem \ref{thm:uniqueness_of_quotients}.
\end{proof}

Below we generalize \cite[Lemma 2.2]{izumino_quotient_bdd} to the setting of von Neumann algebras from which a weaker version of our Theorem \ref{thm:uniqueness_of_quotients} on uniqueness of quotient representation may be derived. This result plays a crucial role in establishing (see Definition \ref{defn:Phi_aff}) that the set of affiliated operators is a functorial construction.

\begin{lem}[{cf. \cite[Lemma 2.2]{izumino_quotient_bdd}}]
\label{lem:ext_maff}
\textsl{
Let $\kM$ be a von Neumann algebra acting on the Hilbert space $\sH$. For $i = 1, 2$, let $A_i, B_i$ be bounded operators in $\kM$ with $\nul (B_i) \subseteq \nul (A_i)$. Then $A_2/B_2$ is an extension of $A_1/B_1$ if and only if there is an operator $C \in \kM$ such that, $A_1 = A_2 C, B_1 = B_2 C$.
}
\end{lem}
\begin{proof}
For bounded (everywhere-defined) operators $A', B'$ acting on a Hilbert space $\sH'$, recall the notation $\fS_{A', B'}$ from \S \ref{sec:prelim}. If $\nul (B') \subseteq \nul (A')$, it is straightforward to see that,
\begin{equation}
\label{eqn:graph_formula}
\graph{A'/B'} 
= \big\{ (B'x, A'x): x \in \sH' \big\}
= \ran(\fS_{A',B'}).
\end{equation}
Thus we have the following equivalences, 
\begin{equation}
    \begin{aligned}
    A_1/B_1 \subseteq A_2/B_2 &\iff \graph{A_1/B_1} \subseteq \graph{A_2/B_2}\\
    &\iff \ran(\fS_{A_1, B_1}) \subseteq \ran(\fS_{A_2, B_2} ). 
    \end{aligned}
\label{eqn:grph_ran}
\end{equation}
\vskip 0.03in

\noindent ($\Longrightarrow$) Assume that $A_1/B_1 \subseteq A_2/B_2$. From Theorem \ref{thm:douglas_vNa} in the context of the von Neumann algebra $M_2(\kM)$ and equation (\ref{eqn:grph_ran}), there is a matrix $\mathfrak{X} \in M_2(\kM)$ such that $$\fS_{A_1, B_1} = \fS_{A_2, B_2} \mathfrak{X}.$$ 
Choosing $C = \mathfrak{X}_{11}$ which lies in $\kM$, the forward implication in the assertion follows.
\vskip 0.05in

\noindent ($\Longleftarrow$) Conversely, if $A_1 = A_2 C, B_1 = B_2 C$, for some $C \in \kM$, then $$\fS_{A_1, B_1} = \fS_{A_2, B_2} (C \oplus 0_{\sH}).$$
Again by using Theorem \ref{thm:douglas_vNa} in the context of $M_2(\kM)$ and equation (\ref{eqn:grph_ran}), we conclude that $A_1/B_1 \subseteq A_2/B_2$.
\end{proof}

\subsection{Algebraic Structure of $\afm$}
\label{subsec:alg_str_m_aff}
In this subsection, we equip the set of affiliated operators with the algebraic structure given by two binary operations, the {\it sum} and the {\it product}, and two unary operations, the {\it Kaufman-inverse} and the {\it adjoint}. Thus from Remark \ref{rem:distributivity_operators}, we note that $(\afm; + , \cdot)$ is a near-semiring (see Definition \ref{def:near-semiring}).

\begin{lem}
\label{lem:aff_times_bdd_op}
\textsl{
Let $\kM$ be a von Neumann algebra.
\begin{itemize}
    \item[(i)] For $A, B \in \kM$,  $A B^{\dagger}$ is in $\afm$. 
    \item[(ii)] Let $A, B, C \in \kM$ such that $\ran(C) \subseteq \ran(B)$ and $\nul (B) \subseteq \nul (A)$. Then $$(A/B) C \in \kM.$$
\end{itemize}
}
\end{lem}
\begin{proof}
\begin{itemize}[leftmargin=0.8cm, itemsep=0.2cm]
    \item[(i)] Note that $\nul (B) = \nul \big( I - \nulN(B) \big) \subseteq \nul \left( A \big( I - \nulN(B) \big) \right)$ and $\ran(B^{\dagger}) \subseteq \nul (B)^{\perp}$. Thus, $$AB^\dagger = A \big( I-\nulN(B) \big) B^\dagger = \big( A (I-\nulN(B) \big)/ B \in \afm.$$
    \item[(ii)] Let $T = A/B$ and $C\in \kM$ as given. Since $\ran(C) \subseteq \dom(T)= \ran(B)$ we can see that $B^\dagger C$ lies in $\kM$ (see Proposition \ref{prop:douglas_soln}). Hence $TC = (A/B) C = A(B^\dagger C)\in \kM$. \qedhere
\end{itemize}
\end{proof}

The first step is to note that $\afm$ is closed under sums and products of operators as defined in Definition \ref{def:str_sum_pdt}. This generalizes \cite[Theorem 3.1, 3.2]{izumino_quotient_bdd} to the setting of von Neumann algebras.

\begin{thm}
\label{thm:m_aff_monoid}
\textsl{
Let $\kM$ be a von Neumann algebra and $T_1, T_2 \in \afm$. Let $A_1, B_1, A_2, B_2 \in \kM$ be such that $T_1 = A_1/B_1, T_2 = A_2/B_2$. Then
\begin{itemize}
    \item[(i)] $T_1 + T_2$ is in $\afm$ with quotient representation $A/B$, where 
\begin{align*}
  A &= (A_1/B_1) (B_1 \boxdot B_2) + (A_2/B_2)(B_1 \boxdot B_2),\\
  B &= B_1 \boxdot B_2.
\end{align*}
    \item[(ii)] $T_1 T_2 $ is in $\afm $ with quotient representation $A/B$, where 
\begin{align*}
  A &= (A_1/B_1) (A_2 (A_2\invran B_1)),\\
  B &= B_2 (A_2\invran B_1).
\end{align*}
\end{itemize}
}
\end{thm}
\begin{proof}

\begin{itemize}[leftmargin=0.8cm, itemsep=0.2cm]
    \item[(i)] Using Lemma \ref{lem:aff_times_bdd_op}-(ii), the operators $A, B$ lie in $\kM$ and it is straightforward to see that $\nul (B) \subseteq \nul(A)$. 
Note that, $$\dom(T_1 + T_2) = \dom(T_1) \cap \dom(T_2) = \ran(B_1) \cap \ran(B_2) = \ran(B_1 \boxdot B_2) = \ran(B).$$ 
Since $T_i B_i = A_i \, (i= 1, 2)$ we have
\begin{align*} 
(T_1 + T_2)B &= (T_1 + T_2) (B_1 \boxdot B_2) =  T_1 (B_1 \boxdot B_2) +  T_2 (B_1 \boxdot B_2)\\
&= T_1 B_1 B_1 ^{\dagger} (B_1 \boxdot B_2) + T_2 B_2 B_2 ^{\dagger} (B_1 \boxdot B_2)\\
&=  A_1 B_1 ^{\dagger} (B_1 \boxdot B_2) + A_2 B_2 ^{\dagger} (B_1 \boxdot B_2) \\
&= A.
\end{align*}
Thus $T_1 + T_2 = A/B$.
    \item[(ii)] From Remark \ref{rem:aff_lat}, note that $A_2 \invran B_1 \in \kM$, and, 
\begin{equation*}
    A_2^{-1} \big( \ran(B_1) \big) = \ran(A_2 \invran B_1),
\end{equation*}
which implies that $$\ran \left( A_2(A_2\invran B_1) \right) = A_2 \ran(A_2 \invran B_1) \subseteq  \ran(B_1).$$ Hence by Lemma \ref{lem:aff_times_bdd_op}-(ii), it follows that $A = (A_1/B_1) (A_2 (A_2 \invran B_1)) \in \kM$. Clearly, $B\in \kM$. Furthermore,
\begin{align*}
\dom\left( T_1T_2  \right) &= \left\{ x \in \dom(T_2) : T_2x \in \dom(T_1) \right\} = \left\{ B_2 x : A_2x \in \ran(B_1), x\in \sH \right\}\\
&= B_2\left( A_2^{-1} \big( \ran(B_1) \big) \right)  
= B_2\left( \ran(A_2\invran B_1) \right) = \ran \left( B_2 (A_2\invran B_1) \right)\\
&= \ran(B).    
\end{align*}

A straightforward algebraic computation shows that,
\begin{align*} (T_1T_2)B = T_1 T_2 B_2 (A_2\invran B_1) =  T_1 A_2 (A_2\invran B_1) = A,
\end{align*}
whence $T_1 T_2 = A/B$. \qedhere
\end{itemize}
\end{proof}

Below we show that $\afm$ is closed under the unary operation, Kaufman-inverse, and give a quotient representation of $T^\dagger$ for $T\in \afm$.

\begin{thm}
\label{thm:m_aff_is_kau_closed}
\textsl{
Let $\kM$ be a von Neumann algebra acting on the Hilbert space $\sH$, and let $T \in \afm$. Let $A/B$ be a quotient representation of $T$ with $A, B \in \kM$ satisfying $\nul (A) \subseteq \nul (B)$. Then
\begin{itemize}
    \item[(i)] $\raR(T), \nulN(T)$ are both contained in $\kM$. In fact, $\nulN(T) = \raR \big( B\nulN(A) \big)$.
    \item[(ii)] $T^{\dagger}$ is contained in $\afm$ with quotient representation given by 
    \[
    T^{\dagger} = \Big( \big( I- \nulN(T) \big) \, B \Big)/A.
    \]
    \item[(ii)] $T$ is one-to-one if and only if $\nul (A) = \nul (B)$, in case of which $T^{-1}$ lies in $\afm$ with quotient representation $B/A$. 
\end{itemize}
}
\end{thm}
\begin{proof} 
\begin{itemize}[leftmargin=0.8cm, itemsep=0.2cm]
    \item[(i)] Clearly $\raR(T) = \raR(A)$, which  lies in $\kM$. Since $T$ is given by the mapping $Bx \mapsto Ax$ ($x \in \sH$), note that \begin{equation}
\label{eqn:T_null}
\nul(T) = \{Bx:Ax = 0, x\in \sH\} = B\Big(\nul(A)\Big).
\end{equation}
Thus $\nulN (T) = \raR \big( B\nulN(A) \big)$, which lies in $\kM$.

\item[(ii)] From equation (\ref{eqn:T_null}), note that the mapping $Ax \mapsto \big( I-\nulN(T) \big)Bx$ ($x \in \sH$) is well-defined and in fact, gives the Kaufman inverse of $T$ (see Definition \ref{def:kaufman_inv}). Thus,
\[
T^{\dagger} = \Big(\big( I-\nulN(T) \big) \; B\Big)/A.
\]
From this expression it is also clear that if $T\in \afm$, then we also have $T^\dagger \in\afm$, that is, $\afm$ is closed under Kaufman-inverse.

\item[(iii)] Keeping in mind the nullspace condition $\nul(B) \subseteq \nul(A)$ and equation (\ref{eqn:T_null}), we have the following equivalences, 

\begin{align*}
T \text{ is one-to-one } &\Longleftrightarrow \nul (T) = \{ 0 \}_{\sH} \Longleftrightarrow B\big( \nul(A) \big) = \{ 0 \}_{\sH}\\
&\Longleftrightarrow \nul(A) \subseteq \nul(B) \Longleftrightarrow \nul(A) = \nul(B).
\end{align*}
When $T$ is one-to-one, it is clear that $T^{-1}$ is given by the mapping $Ax \mapsto Bx$ ($x \in \sH$) so that $T^{-1} = B/A$. \qedhere
\end{itemize}
\end{proof}

The adjoint of an unbounded operator (not necessarily densely-defined) is defined in Definition \ref{def:adj_not_dd_op}. Below we show that $\afm$ is closed under adjoint. 
\begin{thm}
\label{thm:adj_of_aff_opes}
\textsl{
Let $A, B\in \kM$ with $\nul(B) \subseteq \nul(A)$. Then
    \[
        (A/B)^* = (B^*)^\dagger A^*.
    \]
Thus for $T \in \afm$, its adjoint, $T^*$, also lies in $\afm$. In other words, $\afm$ is $*$-closed.
}
\end{thm}

\begin{proof} We prove the assertion in two steps below.

\noindent \textbf{Claim 1.} $\dom \left( (A/B)^* \right) 
    = (A^*)^{-1}\Big( \ran(B^*) \Big) 
    = \dom\left( (B^*)^\dagger A^* \right).$
\vskip 0.02in
\noindent {\it Proof of Claim 1:} Note that
$
    \dom\left( (B^*)^\dagger A^* \right)= (A^*)^{-1}\Big( \ran(B^*) \Big),
$
and from Definition \ref{def:adj_not_dd_op},
\begin{align}
    \dom((A/B)^*) &= \big\{ y \in \sH:  x \mapsto \ipdt{AB^{\dagger}x}{y} \in \C \text{ is continuous on } \dom \big( AB^{\dagger} \big)  \big\}  \nonumber \\
    &= \big\{ y \in \sH:  Bx \mapsto \ipdt{Ax}{y} \in \C \text{ is continuous on } \ran(B) \big\}.  
    \label{eqn:dom_adj}
\end{align}

Let $y\in \dom\left( (B^*)^\dagger A^* \right)$. Then there exists $z \in \sH$ such that $A^* y = B^* z$. For all $x \in \sH$, we have, $$\ipdt{Ax}{y}= \ipdt{x}{A^*y}=\ipdt{x}{B^*z}=\ipdt{Bx}{z}.$$ Since the mapping, $Bx \mapsto \ipdt{Bx}{z}(=\ipdt{Ax}{y})$, is continuous on $\ran(B)$, by (\ref{eqn:dom_adj}) we have $y\in \dom((AB^\dagger)^*)$. Thus, 
\begin{equation}
\label{eqn:adj_of_maff1}
\dom\left( (B^*)^\dagger A^* \right) \subseteq \dom\left( (AB^\dagger)^*\right).
\end{equation}

For the converse, suppose $y \in \dom((AB^\dagger)^*)$ so that the mapping, $Bx \mapsto \ipdt{Ax}{y}$, is continuous. By the Riesz representation theorem, there exists a $z\in \overline{\ran(B)}$ such that $\ipdt{Ax}{y}=\ipdt{Bx}{z}$ for all $x\in \sH$. Thus $A^*y = B^*z$ and $y\in (A^*)^{-1}\Big( \ran(B^*) \Big)$. Thus, 
\begin{equation}
\label{eqn:adj_of_maff2}
\dom\left( (AB^\dagger)^*\right) \subseteq \dom\left( (B^*)^\dagger A^* \right).
\end{equation}

From (\ref{eqn:adj_of_maff1}) and (\ref{eqn:adj_of_maff2}), we conclude that $\dom\left( (AB^\dagger)^*\right) = \dom\left( (B^*)^\dagger A^* \right).
$
\vspace{0.2cm}

\noindent \textbf{Claim 2.} $(A/B)^* y = (B^*)^\dagger A^* y$ for every vector $y$ in $(A^*)^{-1}\Big( \ran(B^*) \Big)$.
\vskip 0.02in
\noindent {\it Proof of Claim 2.} Let $z\in \sH$ be such that $A^*y = B^*z$. Then, $$(B^*)^\dagger A^* y = (B^*)^\dagger B^* z = \big( I - \nulN(B^*) \big) z =  \raR(B)z \in \overline{\ran(B)} .$$
 From the definition of $T^*$, it is clear that $\ran(T^*)\subseteq \overline{\dom(T)}$ whence 
\[
(AB^\dagger)^*y \in \ran\left( (AB^\dagger)^* \right)\subseteq \overline{\ran(B)}.
\]
Thus the vector $v:=(AB^\dagger)^*y - (B^*)^\dagger A^*y$ is in $\overline{\ran(B)}$. For every vector $x\in \sH$, we have 
\begin{align*}
\ipdt{(AB^\dagger)^*y}{Bx} &= \ipdt{y}{AB^{\dagger}Bx} = \ipdt{y}{Ax} = \ipdt{A^*y}{x}= \ipdt{B^*z}{x}  \\
&= \ipdt{z}{Bx} = \ipdt{z}{\raR(B)Bx}  = \ipdt{\raR(B)z}{Bx} \\
&= \ipdt{(B^*)^\dagger A^*y}{Bx},
\end{align*}
which implies that $v$ is orthogonal to $\overline{\ran(B)}$. Thus $v=0$, that is, $(AB^\dagger)^* y = (B^*)^\dagger A^* y$.

Since $(B^*)^\dagger, A^* \in \afm$, by Theorem \ref{thm:m_aff_monoid}-(ii), we conclude that $(A/B)^* = (B^*)^\dagger A^*$ is in $\afm$. 
\end{proof}

\subsection{Functoriality of the Construction $\kM \mapsto \afm$}
\label{subsec:functoriality_m_to_afm}

For a unital normal $*$-homomorphism $\Phi$ between represented von Neumann algebras $\kM$ and $\kN$, we define a mapping $\Phi_{\tn\aff}$ between their respective near-semirings of affiliated operators, $\afm$ and $\afn$, and show that it respects the four algebraic operations discussed in \S \ref{subsec:alg_str_m_aff}; In particular, it is a near-semiring homomorphism. Thus the construction $$\kM \mapsto \afm, \Phi \mapsto \Phi_{\tn\aff},$$ is functorial. In Theorem \ref{thm:uniqueness_phi_aff}, we argue that $\Phi_{\tn\aff}$ is the canonical extension of $\Phi$ to $\afm$ (in an algebraic sense) by showing that it is the unique extension of $\Phi$ which respects product and Kaufman-inverse; in Theorem \ref{thm:m_aff_monoid}-(i), we have already noted that $\Phi_{\tn\aff}$ respects sum and adjoint.

\begin{prop}
\label{prop:phi_aff_is_well_def}
\textsl{
Let $(\kM; \sH)$ and $(\kN; \sK)$ be represented von Neumann algebras and $\Phi:\kM\to \kN$ be a unital normal $*$-homomorphism. For $i = 1, 2$, let $A_i, B_i$ be operators in $\kM$ with $\nul (B_i) \subseteq \nul (A_i)$. Then we have the following:
\begin{itemize}
    \item[(i)] $\nul \big( \Phi(B_i) \big) \subseteq \nul \big( \Phi(A_i) \big)$;
    \item[(ii)] If $A_1/B_1 \subseteq A_2/B_2$, then $\Phi(A_1)/\Phi(B_1) \subseteq \Phi(A_2)/\Phi(B_2)$.
    \item[(iii)] If $A_1/B_1 = A_2/B_2$, then $\Phi(A_1)/\Phi(B_1) =\Phi(A_2)/\Phi(B_2).$
\end{itemize}
}
\end{prop}
\begin{proof}
\begin{itemize}[leftmargin=0.8cm, itemsep=0.2cm]
    \item[(i)] Follows from Lemma \ref{lem:morph_prop}. Thus the quotients in parts (ii) and (iii) make sense. 
    \item[(ii)] By Lemma \ref{lem:ext_maff}, there is an operator $C \in \kM$ such that $A_1 = A_2C, B_1 = B_2 C$. Thus $\Phi(A_1) = \Phi(A_2) \Phi(C), \Phi(B_1) = \Phi(B_2) \Phi(C)$. Again from Lemma \ref{lem:ext_maff} in the context of $\kN$, the desired conclusion follows.
    \item[(iii)] Immediately follows from part (ii). \qedhere
\end{itemize}
\end{proof}

\begin{definition}
\label{defn:Phi_aff}
Let $(\mathscr{M}; \sH), (\mathscr{N}; \sK)$ be represented von Neumann algebras and $\Phi : \mathscr{M} \to \mathscr{N}$ be a unital normal $*$-homomorphism. Let $\Phi_{\tn\aff} : \afm \to \afn$ be the mapping, which for every pair of operators $A, B \in \kM$ with $\nul (B) \subseteq \nul (A)$, sends $A/B$ to $\Phi(A)/\Phi(B)$; by Proposition \ref{prop:phi_aff_is_well_def}, this mapping makes sense and is well-defined. In short,
\[
\Phi_{\tn\aff}(A/B) := \Phi(A)/\Phi(B).
\]
\end{definition}

\begin{lem}
\label{lem:phi_pres_a_b_c}
\textsl{
Let $(\mathscr{M}; \sH), (\mathscr{N}; \sK)$ be represented von Neumann algebras and $\Phi : \mathscr{M} \to \mathscr{N}$ be a unital normal $*$-homomorphism. Let $A, B, C\in \kM$ such that $\ran(C)\subseteq \ran(B)$ and $\nul(B) \subseteq \nul(A)$. Then $(A/B) C \in \kM$, $\big( \Phi(A) /\Phi(B) \big) \Phi(C) \in \kN$, and we have,
\[
    \Phi\big( (A/B) C \big) = \big( \Phi(A) /\Phi(B) \big) \Phi(C).
\]
}
\end{lem}
\begin{proof}
From Lemma \ref{lem:aff_times_bdd_op}-(ii), we know that $(A/B) C \in \kM$. From Lemma \ref{lem:phi_pres_douglas}-(i) and Proposition \ref{prop:phi_aff_is_well_def}-(i), it is clear that $\ran \big( \Phi(C) \big) \subseteq \ran \big( \Phi(B) \big)$ and $\nul \big( \Phi(B) \big) \subseteq \nul \big( \Phi(A) \big)$. Again  using Lemma \ref{lem:aff_times_bdd_op}-(ii), we see that $\big( \Phi(A) /\Phi(B) \big) \Phi(C) \in \kN$. The Douglas factorization lemma (Theorem \ref{thm:douglas_vNa}) tells us that $B^\dagger C \in \kM$. Thus, 
    \begin{align*}
        \Phi \big( (A/B) C \big) &= \Phi( AB^\dagger C ) 
            = \Phi(A) \Phi(B^\dagger C) \\
            &= \Phi(A) \Phi(B)^\dagger \Phi(C), \text{ by Lemma \ref{lem:phi_pres_douglas}-(ii),} \\
            &= \big( \Phi(A) / \Phi(B) \big) \Phi(C). \qedhere
    \end{align*}
\end{proof}

\begin{thm}
\label{thm:phi_monoid_morph}
\textsl{
Let $(\mathscr{M}; \sH), (\mathscr{N}; \sK)$ be represented von Neumann algebras and $\Phi : \mathscr{M} \to \mathscr{N}$ be a unital normal $*$-homomorphism. Let $\Phi_{\tn\aff} : \afm \to \afn$ be the extension of $\Phi$ as given in Definition \ref{defn:Phi_aff}. Then for $T_1, T_2 \in \afm$, we have,
\begin{itemize}
    \item[(i)] $\Phi_{\tn\aff}(T_1+T_2) = \Phi_{\tn\aff}(T_1) + \Phi_{\tn\aff}(T_2)$;
    \item[(ii)] $\Phi_{\tn\aff}(T_1 T_2) = \Phi_{\tn\aff}(T_1) \;\Phi_{\tn\aff}(T_2)$;
    \item[(iii)] $\Phi_{\tn\aff}(T_1 ^{\dagger}) = \Phi(T_1)^{\dagger}$, and in particular, if $T \in \afm$ is non-singular, then $\Phi_{\tn\aff}(T)$ is non-singular with $\Phi_{\tn\aff}(T^{-1}) = \Phi_{\tn\aff}(T)^{-1}$;
    \item[(iv)] $\Phi_{\tn\aff}(T_1 ^*) = \Phi_{\tn\aff}(T_1)^*$.
\end{itemize}
In other words, $\Phi_{\tn\aff}$ preserves sum, product, Kaufman-inverse, adjoint.
}
\end{thm}
\begin{proof} Let $T_i = A_i/B_i$, where $A_i, B_i \in \kM$ for $i = 1, 2$. Then we have 
\[
    \Phi_{\tn\aff}(T_i) = \Phi(A_i)/ \Phi(B_i), ~ (i = 1, 2).
\]
\begin{itemize}[leftmargin=0.8cm, itemsep=0.2cm]
    \item[(i)] From Theorem \ref{thm:m_aff_monoid}-(i), $T_1 + T_2= A/B$, where \begin{align*}
    A &= (A_1/B_1) (B_1\boxdot B_2) + (A_2/B_2) (B_1\boxdot B_2),\\
    B &= B_1\boxdot B_2.
    \end{align*} Using Lemma \ref{lem:phi_pres_a_b_c} and Lemma \ref{lem:phi_pres_boxes}, it is easy to see that,
    \begin{equation*}
        \begin{aligned}
            \Phi(A) &= \Big( \Phi(A_1) / \Phi(B_1) \Big) \Big( \Phi(B_1) \boxdot \Phi(B_2) \Big) + \Big( \Phi(A_2) / \Phi(B_2) \Big) \Big( \Phi(B_1) \boxdot \Phi(B_2) \Big),\\
            \Phi(B) &= \Phi(B_1) \boxdot \Phi(B_2).
        \end{aligned}
    \end{equation*}
    Now applying Theorem \ref{thm:m_aff_monoid}-(i) on the two quotients $\Phi(A_i)/\Phi(B_i)$, for $i = 1, 2$ we get,
    \[
    \Phi_{\tn\aff}(T_1) + \Phi_{\tn\aff}(T_2) = \Big(\Phi(A_1) / \Phi(B_1)\Big) + \Big(\Phi(A_2) / \Phi(B_2)\Big) = \Phi(A) / \Phi(B).
    \]
    Therefore we have,
    \begin{align*}
        \Phi_{\tn\aff}(T_1 + T_2) = \Phi_{\tn\aff}(A/B) 
                    = \Phi(A) / \Phi(B) = \Phi_{\tn\aff}(T_1) + \Phi_{\tn\aff}(T_2).
    \end{align*}    

    \item[(ii)] The proof is almost identical to that of part (i) using Theorem \ref{thm:m_aff_monoid}-(ii) in lieu of Theorem \ref{thm:m_aff_monoid}-(i).

    \item[(iii)] From Theorem \ref{thm:m_aff_is_kau_closed}-(ii), we have the following quotient representation of $T_1^{\dagger}$, $$T_1^\dagger = \Big( \big( I-\nulN(T_1) \big) B_1\Big)/A_1.$$ Since by Theorem \ref{thm:m_aff_is_kau_closed}-(i), we have $\nulN(T_1) = \raR \big( B_1 \nulN(A_1) \big)$, Lemma \ref{lem:morph_prop} tells us that $\Phi \big( \nulN(T_1) \big) = \nulN \big( \Phi_{\tn\aff}(T_1) \big)$; thus,
    \begin{align*}
        \Phi_{\tn\aff}(T_1^\dagger) 
        = \Big( \left( I - \nulN \big( \Phi_{\tn\aff}(T_1) \big) \right) \Phi(B_1) \Big) / \Phi(A_1)
        = \Phi_{\tn\aff}(T_1)^\dagger.
    \end{align*}
    
     For a non-singular (that is, one-to-one) affiliated operator $T'$, clearly $T'^{-1} = T'^{\dagger}$. From Theorem \ref{thm:m_aff_is_kau_closed}-(iii) and Lemma \ref{lem:morph_prop}, we have,
    \begin{align*}
    T_1 \text{ is one-to-one } &\implies \nulN(A_1) = \nulN(B_1) \implies \nulN \big( \Phi(A_1) \big) = \nulN \big( \Phi(B_1) \big)\\
    &\implies \Phi_{\tn\aff}(T_1) = \Phi(A_1)/\Phi(B_1) \text{ is one-to-one.}
    \end{align*}
   Thus, if $T_1$ is one-to-one, then $\Phi_{\tn\aff}(T_1$ is one-to-one and we have $\Phi_{\tn\aff}(T_1^{-1}) = \Phi_{\tn\aff}(T_1)^{-1}$.

   \item[(iv)] By Theorem \ref{thm:adj_of_aff_opes}, $T_1^* = (B_1^*)^\dagger A_1^*$. Hence using part (ii) and (iii), we get,
    \begin{align*}
        \Phi_{\tn\aff}(T_1^*) 
        &= \Phi_{\tn\aff}\left( (B_1^*)^\dagger A_1^* \right)
        \overset{\text{(ii)}}{=} \Phi_{\tn\aff}\left( (B_1^*)^\dagger \right) \Phi_{\tn\aff}(A_1^*)\\
        &\overset{\text{(iii)}}{=} \left(\Phi(B_1)^*\right)^\dagger \Phi(A_1)^*
        = \Phi_{\tn\aff}(T_1)^*. \qedhere
    \end{align*}
\end{itemize}
\end{proof}

In the theorem below, we note that $\Phi_{\tn\aff}$ is the canonical extension of $\Phi$ in an algebraic sense.

\begin{thm}
\label{thm:uniqueness_phi_aff}
\textsl{
Let $(\mathscr{M}; \sH), (\mathscr{N}; \sK)$ be represented von Neumann algebras and $\Phi : \mathscr{M} \to \mathscr{N}$ be a unital normal $*$-homomorphism.Let $\Phi_{\tn\aff} : \afm \to \afn$ be the extension of $\Phi$ as given in Definition \ref{defn:Phi_aff}.  Then  $\Phi_{\tn\aff}$ is the unique mapping from $\afm$ to $\afn$ which extends $\Phi$, and respects product and Kaufman-inverse.
}
\end{thm}
\begin{proof} 
Let $\Psi$ be another such extension of $\Phi$. Then for $T=AB^\dagger$ in $\afm$, 
\[
\Psi(T) = \Psi(A) \Psi(B^{\dagger}) = \Psi(A) \Psi(B)^{\dagger} = \Phi(A) \Phi(B)^{\dagger} = \Phi_{\tn\aff}(AB^{\dagger}) = \Phi_{\tn\aff}(T). \qedhere
\]
\end{proof}

We rephrase Proposition \ref{prop:phi_aff_is_well_def}-(ii) in terms of $\Phi_{\tn\aff}$ and note that $\Phi_{\tn\aff}$ is well-behaved with respect to extensions of operators. In \S\ref{sec:misc_alg_props} and \S \ref{sec:krein_friedrichs}, we will see other algebraic properties that are preserved by $\Phi_{\tn\aff}$.

\begin{thm}
\label{thm:phi_pres_ext}
\textsl{
Let $(\mathscr{M}; \sH), (\mathscr{N}; \sK)$ be represented von Neumann algebras and $\Phi : \mathscr{M} \to \mathscr{N}$ be a unital normal $*$-homomorphism. Let $\Phi_{\tn\aff} : \afm \to \afn$ be the extension of $\Phi$ as given in Definition \ref{defn:Phi_aff}. Let $T, T_0$ be operators in $\afm$ with $T_0$ being an extension of $T$. Then $\Phi_{\tn\aff}(T_0)$ is an extension of $\Phi_{\tn\aff}(T)$.
}
\end{thm}
\begin{proof}
Immediate from Proposition \ref{prop:phi_aff_is_well_def}-(ii).
\end{proof}

In our pursuit of functoriality, we have taken primarily an algebraic viewpoint of the set of affiliated operators. An alternative geometric viewpoint becomes available by looking at the graphs of affiliated operators. Let $T \in \afm$. From equation (\ref{eqn:graph_formula}), clearly $\graph{T}$ is in $\affs{M_2(\kM)}$. Theorem \ref{thm:graph_phi} below shows that it is possible to provide a geometric definition of $\Phi_{\tn\aff}$ by considering the mapping induced by $\Phi$ (more precisely, $\Phi_{(2)}$) on the affiliated subspaces for $M_2(\kM)$.
    
\begin{thm}
\label{thm:graph_phi}
\textsl{
Let $(\mathscr{M}; \sH), (\mathscr{N}; \sK)$ be represented von Neumann algebras and $\Phi : \mathscr{M} \to \mathscr{N}$ be a unital normal $*$-homomorphism. Let $\Phi_{\tn\aff} : \afm \to \afn$ be the extension of $\Phi$ as given in Definition \ref{defn:Phi_aff}. Then the following diagram commutes.
}
\[
\begin{tikzcd}
\afm \arrow{r}{\tn{Graph}} \arrow[swap]{d}{\Phi_{\tn\aff}} & \tn{Aff}_{s}\Big( M_2(\kM) \Big) \arrow{d}{\big( \Phi_{(2)} \big)_{\tn\aff}^s} \\
\afn \arrow{r}{\tn{Graph}} & \tn{Aff}_{s}\Big( M_2(\kN) \Big)
\end{tikzcd}
\]
\textsl{
In other words, $\Phi_{\tn\aff}(T)$ is the unique operator in $\afn$ which satisfies,
\[
\graph{\Phi_{\tn\aff}(T)} = (\Phi_{(2)})^{s}_{\tn\aff} \left( \graph{T} \right).
\]
}
\end{thm}
\begin{proof}
Recall the notation $\fS_{A', B'}$ from \S \ref{sec:prelim}. Let $T \in \afm$ and $A/B$ be a quotient representation of $T$ for $A, B \in \kM$ with $\nul (B) \subseteq \nul (A)$. As observed before in Proposition \ref{prop:phi_aff_is_well_def}, $\Phi(A) / \Phi(B)$ is a quotient representation of $\Phi_{\tn\aff}(T)$. Since $\fS_{\Phi(A), \Phi(B)} = \Phi_{(2)}(\fS_{A, B})$, using equation (\ref{eqn:graph_formula}) we have,
\begin{align*}
    \graph{\Phi_{\tn\aff}(T)} &\overset{(\ref{eqn:graph_formula})}{=}
    \ran\left( \fS_{\Phi(A), \Phi(B) }\right)
    = \ran \left( \Phi_{(2)} (\fS_{A, B} )\right)\\
    &= (\Phi_{(2)})^{s}_{\aff} \big( \ran \left( \fS_{A, B} \right)
    \big) \overset{(\ref{eqn:graph_formula})}{=} (\Phi_{(2)})^{s}_{\aff} \big( \graph{T} \big).\qedhere
\end{align*}
\end{proof}

\section{Closed Affiliated Operators}%
\label{sec:closed_aff_operators}%

In this section, our main goal is to establish that our definition of affiliation (Definition \ref{defn:M_aff}) expands the scope in which the term is traditionally used. Let $\kM$ be a von Neumann algebra acting on the Hilbert space $\sH$. We show that a \emph{closed} operator $T$ on $\sH$ is MvN-affiliated (see Definition \ref{def:mvn_aff}) if and only if it is affiliated according to our definition. Thus our definition does not affect the traditional usage of the term in the sense of Murray and von Neumann while providing significant algebraic simplicity; there is everything to gain and little to lose by embracing our definition of $\afm$. 

\subsection{Quotient Representation of Closed Affiliated Operators}
We have mentioned some notation below to facilitate our discussion.
\begin{align*}
    \afm &:=\{A/B: A, B \in \kM \text{ satisfying } \nul(B)\subseteq \nul(A) \}\\
    \affc{\kM} &:= \afm \cap \sC(\sH) = \{ T \in \afm : T \text{ is closed} \} \\
    \kM_{\aff} ^{cb} &:= \afm \cap \sC\sB(\sH) =  \{ T \in \afm : T \text{ is closed, and bounded on } \dom(T) \}\\
    \mvnaff{\kM} &:=\{ T \in \sC(\sH): U^*TU = T \text{ for all unitary } U \in \kM' \}.\\
\end{align*}

The following containment relations hold, $$\kM \subseteq \afm^{cb}  \subseteq \afm ^c = \afm^{\text{MvN}}.$$
Note that the last equality is the content of Theorem \ref{thm:quotient}.

\begin{remark}
\label{rem:closed_bdd}
An operator in $\afm$ is {\it closed and bounded} if and only if it is the restriction of an operator in $\kM$ to a closed affiliated subspace for $\kM$.
Note that $$\afm^{cb} = \{ AE^{\dagger} : A, E \in \kM \text{ with } E  \text{ a projection in } \kM\}.$$
\end{remark}

\begin{lem}
\label{lem:char_aff}
\textsl{
Let $\mathscr{M}$ be a von Neumann algebra acting on the Hilbert space $\sH$. A closed operator $T$ on $\sH$ is in $\afm^{\text{MvN}}$ if and only if its characteristic projection, $\rchi(T)$, is in $M_2(\mathscr{M})$.
}
\end{lem}
\begin{proof}
Note that every unitary operator in $M_2(\kM)'$ is of the form 
	$\begin{bmatrix}
		U & 0\\
		0 & U
	\end{bmatrix}$
	for some unitary $U\in \kM'$. 
 
 Let $T \in \afm^{\text{MvN}}$. For a vector $x\in \dom(T)$ and $U\in \kM'$, note that $Ux \in \dom(T)$, and 
	
	\[
		\begin{bmatrix}
			U & 0\\
			0 & U
		\end{bmatrix}
		\begin{bmatrix}
			x\\Tx
		\end{bmatrix}
		=
		\begin{bmatrix}
			Ux\\UTx
		\end{bmatrix}
		=
		\begin{bmatrix}
			Ux\\TUx
		\end{bmatrix}.
	\]
Thus $\text{Graph}(T)$ is invariant under the action of any unitary operator in $M_2(\kM)'$ and by  the double commutant theorem $\rchi(T)\in M_2(\kM)$.
	
Conversely, assume that $T$ is a closed operator on $\sH$ with $\rchi(T)\in M_2(\kM)$. Let $U$ be a unitary in $\kM'$. Then $\rchi(T)$ commutes with 
	$
		\begin{bmatrix}
			U & 0\\
			0 & U
		\end{bmatrix}
	$ so that 
	$
	\begin{bmatrix}
		U & 0\\
		0 & U
	\end{bmatrix}
	$ leaves $\text{Graph}(T)$ invariant, that is, 
	\[
		\begin{bmatrix}
			Ux\\
			UTx
		\end{bmatrix}
		=
		\begin{bmatrix}
			U & 0\\
			0 & U
		\end{bmatrix}
		\begin{bmatrix}
			x\\
			Tx
		\end{bmatrix}
		\in \text{Graph}(T)
		~\text{ for all }
		x \in \dom(T).
	\]
Thus the second coordinate of 
$
	\begin{bmatrix}
		Ux\\
		UTx
	\end{bmatrix}
	$
must be $T(Ux)$ so that $UTx = TUx$ for all $x\in \dom(T)$. Hence $T$ commutes with every unitary $U$ in $\kM '$, that is,  $T\in \kM_{\aff}^{\text{MvN}}$. 
\end{proof}

The following theorem generalizes \cite[Theorem 1.1]{fillmore-williams} and shows that the lattice of affiliated subspaces, $\affs{\kM}$, is precisely the set of domains of MvN-affiliated operators.

\begin{thm}
\label{thm:dom_ran}
\textsl{
Let $\mathscr{M}$ be a von Neumann algebra acting on the Hilbert space $\sH$. Let $\sV$ be a linear subspace of $\sH$. Then the following are equivalent:
\begin{itemize}
    \item[(i)] $\sV$ is the range of an operator in $\mathscr{M}$;
    \item[(ii)] $\sV$ is the range of an operator in $\afm^{\text{MvN}}$;
    \item[(iii)] $\sV$ is the domain of an operator in $\afm^{\text{MvN}}$;
    \item[(iv)] $\sV$ is the domain of an operator in $\afm$.
\end{itemize}
}
\end{thm}

\begin{proof}
\noindent {\large (i) $\Rightarrow$ (ii)}. Straightforward to see.
\vskip 0.08in

\noindent {\large (ii) $\Rightarrow$ (iii)}. Let $T$ be a closed operator in $\afm^{\text{MvN}}$. Then $T^{\dagger}$ is a closed operator in $\afm^{\text{MvN}}$ and $\dom(T^{\dagger}) = \ran(T)$.
\vskip 0.08in

\noindent {\large (iii) $\Rightarrow$ (i)}.
Let $\sV$ be the domain of an operator $T\in \kM_{\aff}^{\text{MvN}}$, that is, $\sV= \dom(T)$. From Lemma \ref{lem:char_aff}, the characteristic projection of $T$, $$\chi(T) = \begin{bmatrix}
E_{11} & E_{12}\\
E_{21} & E_{22}
\end{bmatrix},$$ is contained in $M_2(\kM)$. By equation (\ref{eqn:dom_from_char_mat}) and Remark \ref{rem:aff_lat}, we have $$\sV = \dom(T) = \ran(E_{11}) + \ran (E_{12}) = \ran (E_{11} \boxplus E_{12}),$$
and $E_{11} \boxplus E_{12} \in \kM$ as $E_{11}, E_{12} \in \kM$.
\vskip 0.08in
\noindent {\large (i) $\Leftrightarrow$ (iv)}. If a subspace $\sV \subseteq \sH$ is the range of $B \in \kM$, since $\dom(B^{\dagger}) = \ran(B)$, we have $\sV$ is the domain of $B^{\dagger} \in \afm$. Conversely, if $\sV$ is the domain of an operator $T$ in $\afm$ and $A/B$ is a quotient representation of $T$ for $A, B \in \kM$ with $\nul (B) \subseteq \nul (A)$, then $\sV = \dom(T) = \ran(B)$.
\end{proof}

\begin{thm}[Quotient representation of closed affiliated operators]
\label{thm:quotient}
\textsl{
Let $\mathscr{M}$ be a von Neumann algebra acting on the Hilbert space $\sH$. Let $T$ be a closed operator on $\sH$. 
\begin{itemize}
    \item[(i)]  Then $T$ is affiliated with $\kM$ in the sense of Murray and von Neumann if and only if $T \in \afm$, that is,  $$\afm^{\text{MvN}} = \afm^c.$$
    \item[(ii)] We have $T \in \afm ^c$ if and only if $T= P^{-1}AE^{\dagger}$ for operators $P, A, E \in \kM$ with $P$ being a positive one-to-one operator with dense range, and $E$ a projection. 
\end{itemize}
}
\end{thm}
\begin{proof}
\begin{itemize}[leftmargin=0.8cm, itemsep=0.2cm]
    \item[(i)] Let $T = AB^{\dagger}$ for $A, B \in \kM$, and $U$ be a unitary in $\kM'$. By Lemma \ref{lem:unbdd} $B^\dagger \in \afm^{\text{MvN}}$ and thus $UT = UAB^{\dagger} = AUB^{\dagger} = AB^{\dagger}U = TU$. Hence $T \in \afm^{\text{MvN}}$, and we have,
    $$\afm ^{c} \subseteq \afm^{\text{MvN}}.$$

    Next suppose $T\in \afm^{\text{MvN}}$. By Theorem \ref{thm:dom_ran}, there is an operator $B'$ in $\mathscr{M}$ such that $\ran(B') = \dom(T)$. Then by Lemma \ref{lem:unbdd}, $TB'$ is a bounded operator in $\mathscr{M}$ with $\ran(T) = \ran(TB')$ and $\nul(B') \subseteq \nul(TB')$. Defining $A' := TB'$, we have $T = (TB')B'^{\dagger}=A'B'^\dagger$ with $\dom(T) = \ran(B')$, $\ran(T) = \ran(A')$, and $\nul(B') \subseteq \nul(A')$. Thus $T = A'/B'$ is in $\afm^c$, and we conclude that $$\afm^{\text{MvN}} \subseteq \afm ^{c}.$$

    \item[(ii)] ($\Longleftarrow$) Let $A, P, E$ be as given in the assertion. Note that $P^{-1}$ is a closed operator and $AE^{\dagger} = A|_{\ran(E)}$ is closed and bounded. Thus the operator $P^{-1} A E^{\dagger}$ is closed. Since $P^{-1}, AE^{\dagger} \in \afm$ and $\afm$ is closed under product by Theorem \ref{thm:m_aff_monoid}, $P^{-1} AE^{\dagger}$ lies in $\afm$. In summary, $P^{-1} AE^{\dagger} \in \afm^{c}$.
    \vskip 0.08in

    \noindent($\Longrightarrow$) Let $T\in \afm$ be closed. By Theorem \ref{thm:adj_of_aff_opes}, $T^* \in \afm$ and by Proposition \ref{prop:t_bar}, it is closed and densely-defined. Since $\dom(T^*) \in \affs{\kM}$, using Corollary \ref{cor:ran_aa_star}, we get a positive operator $P \in \kM$ such that $\dom(T^*) = \ran(P)$. Since $T^*$ is densely-defined, $P$ has dense range and by Proposition \ref{prop:rank_nullity}, it has trivial nullspace. Since $T^* \in \afm^{c}$, part (i) tells us that $T^*$ is in $\afm^{\text{MvN}}$. By Lemma \ref{lem:unbdd}, the operator $A' := T^*P$ lies in $\kM$, and we see that $T^* = A'P^{-1} = A'/P$. Let $E$ be the projection onto $\overline{\dom(T)}$. Using Proposition \ref{prop:t_bar}, Theorem \ref{thm:adj_of_aff_opes}, and our hypothesis that $T$ is closed, we see that, 
\[
T = \overline{T} = T^{**}E^\dagger=(A'/P)^* E^{\dagger} =  P^{-1} A'^* E^\dagger. \qedhere
\] 
\end{itemize}
\end{proof}

\subsection{Functoriality of the Construction $\mathscr{M} \mapsto \mathscr{M}_{\textnormal{aff}}^{c}$}
\label{subsec:closed_functoriality}

Since the existing literature about affiliation deals with MvN-affiliated operators, we must discuss what our theory has to say in this context. Throughout this subsection, $\kM, \kN$ are von Neumann algebras and $\Phi : \kM \to \kN$ is a unital normal $*$-homomorphism. We first demonstrate that the restriction of $\Phi_{\tn\aff}$ (see Definition \ref{defn:Phi_aff}) to $\afm ^{\text{MvN}}$ maps into $\afn ^{\text{MvN}}$, and sends densely-defined operators to densely-defined operators. After that, we direct our energies towards convincing the reader that $\Phi_{\tn\aff}$ is the canonical extension of $\Phi$ to $\afm$, from various perspectives.

\begin{thm}
\label{thm:phi_pres_closed}
\textsl{
Let $(\mathscr{M}; \sH), (\mathscr{N}; \sK)$ be represented von Neumann algebras and $\Phi : \mathscr{M} \to \mathscr{N}$ be a unital normal $*$-homomorphism. Let $\Phi_{\tn\aff} : \afm \to \afn$ be the extension of $\Phi$ as given in Definition \ref{defn:Phi_aff}. Let $T \in \afm$.
\begin{itemize}
    \item[(i)] Then $\dom \big( \Phi_{\tn\aff}(T) \big) = \Phi_{\tn\aff}^s \big( \dom (T) \big)$. If $T$ is densely-defined, then so is $\Phi_{\tn\aff}(T)$.
    \item[(ii)] If $T$ is closed and bounded, then so is $\Phi_{\tn\aff}(T)$.
    \item[(iii)] If $T$ is closed, then so is $\Phi_{\tn\aff}(T)$.
    \item[(iv)] If $T$ is pre-closed, then so is $\Phi_{\tn\aff}(T)$. Furthermore, $\overline{T} \in \afm$ and $\Phi_{\tn\aff}(\overline{T}) = \overline{\Phi_{\tn\aff}(T)}$.
\end{itemize}
Thus $\Phi_{\tn\aff}$ maps $\afm^{\text{MvN}}$ into $\afn^{\text{MvN}}$, and in fact, sends densely-defined closed affiliated operators for $\kM$ to densely-defined closed affiliated operators for $\kN$.
}
\end{thm}
\begin{proof}

\begin{itemize}[leftmargin=0.8cm, itemsep=0.2cm]
    \item[(i)] Let $T = A/B$ for $A, B \in \kM$ with $\nul (B) \subseteq \nul (A)$, so that $\Phi_{\tn\aff}(T) = \Phi(A)/\Phi(B)$. From Theorem \ref{thm:lat_morph}-(i), we have
\begin{equation}
\label{eqn:dom_func}
\dom \big( \Phi_{\tn\aff}(T) \big)= \ran\; \Phi(B) = \Phi_{\tn\aff}^s\big( \ran(B) \big) = \Phi_{\tn\aff}^s \big( \dom(T) \big).
\end{equation}

Let $T$ be densely-defined, that is, $\dom(T)$ is dense in $\sH$. From Theorem \ref{thm:lat_morph}-(iv), note that $\Phi_{\tn\aff}^s \big( \dom(T) \big)$ is dense in $\sK$. Equation (\ref{eqn:dom_func}) tells us that $\Phi_{\tn\aff}(T)$ is densely-defined.

    \item[(ii)] If $T$ is closed and bounded, then $T = AE^{\dagger}$ for $A \in \kM$ and $E$ a projection in $\kM$. Thus $\Phi_{\tn\aff}(T) = \Phi(A) \Phi(E)^{\dagger}$ where $\Phi(A) \in \kN$ and $\Phi(E)$ is a projection in $\kN$. By Remark \ref{rem:closed_bdd}, the assertion follows.

    \item[(iii)] By Theorem \ref{thm:quotient}-(ii), there are operators $P \in \kM^{+}, A \in \kM, E \in \kM_{\text{proj}}$ with $P$ quasi-invertible such that $T = P^{-1} A E^{\dagger}$. Theorem \ref{thm:phi_monoid_morph}-(iii) tells us that $\Phi(P)$ is a positive quasi-invertible operator with $\Phi_{\tn\aff}(P^{-1}) = \Phi(P)^{-1}$, and clearly, $\Phi(E)$ is a projection in $\kN$. From Theorem \ref{thm:phi_monoid_morph}-(ii), we have $\Phi_{\tn\aff}(T) = \Phi(P)^{-1}\Phi(A) \Phi(E)^{\dagger}$, so that $\Phi_{\tn\aff}(T)$ is closed.

    \item[(iv)] Let $T$ be pre-closed. By Proposition \ref{prop:kr_I_278}, $\dom(T^*)$ is dense in $\sH$. By Theorem \ref{thm:phi_monoid_morph}-(iv) and part (i) above, $\Phi_{\tn\aff}(T^*) = \Phi_{\tn\aff}(T)^*$ is densely-defined on $\sK$. Again using Proposition \ref{prop:kr_I_278}-(i), we conclude that $\Phi_{\tn\aff}(T)$ is pre-closed. 

    Let $A/B$ be a quotient representation of $T$ for $A, B \in \kM$ with $\nul (B) \subseteq \nul (A)$. 
    Let $E$ denote the projection onto $\overline{\dom(T)}$. Since $\dom(T) = \ran(B)$, clearly $E = \raR(B)$. Since $\Phi(A)/\Phi(B)$ is a quotient representation of $\Phi_{\tn\aff}(T)$ and $\Phi(E) = \raR \big( \Phi(B) \big)$ (by Lemma \ref{lem:morph_prop}), we note that $\Phi(E)$ is the projection onto $\overline{\dom(\Phi_{\tn\aff}(T))}$. From Proposition \ref{prop:t_bar} and Theorem \ref{thm:phi_monoid_morph}, we have 
    \[
    \Phi_{\tn\aff}(\overline{T}) =  \Phi_{\tn\aff}(T^{**} E^{\dagger}) = \Phi_{\tn\aff}(T)^{**} \Phi(E)^{\dagger} = \overline{\Phi_{\tn\aff}(T)}.\qedhere
    \]
\end{itemize}
\end{proof}

Let $\kM$ be a finite von Neumann algebra. The set of densely-defined operators in $\afm^{\text{MvN}}$ is a $*$-algebra under strong-sum, strong-product, and adjoint. For the sake of brevity, we denote this $*$-algebra by $\afm^{dd-c}$. In fact, it is a monotone-complete partially-ordered $*$-algebra under the order structure given by the cone of positive operators in $\afm^{dd-c}$ (see \cite[Proposition 4.21]{Nayak_2021}). Let $\kN$ be another finite von Neumann algebra and $\Phi : \kM \to \kN$ be a unital normal $*$-homomorphism. In \cite[Theorems 4.9, 4.24]{Nayak_2021} and \cite[Proposition 5.14]{hiroshi_2012}, it is shown that $\Phi$ has a unique extension to a unital normal $*$-homomorphism between $\afm^{dd-c}$ and $\afn^{dd-c}$. From the corollary below, we observe that this extension is nothing but the restriction of our $\Phi_{\tn\aff}$ to $\afm^{dd-c}$. This shows the compatibility of $\Phi_{\tn\aff}$ with the canonical extension of $\Phi$ in the context of finite von Neumann algebras.

\begin{cor}
\label{cor:phi_pres_st_sum_pdt}
\textsl{
Let $(\mathscr{M}; \sH), (\mathscr{N}; \sK)$ be represented von Neumann algebras and $\Phi : \mathscr{M} \to \mathscr{N}$ be a unital normal $*$-homomorphism. Let $\Phi_{\tn\aff} : \afm \to \afn$ be the extension of $\Phi$ as given in Definition \ref{defn:Phi_aff}. Let $T_1, T_2$ be closed operators affiliated with $\kM$. 
\begin{itemize}
    \item[(i)] If $T_1 + T_2$ is pre-closed, then $\Phi_{\tn\aff}(T_1) + \Phi_{\tn\aff}(T_2)$ is pre-closed, and $$\Phi_{\tn\aff}(T_1 \;\hat{+}\; T_2) = \Phi_{\tn\aff}(T_1) \;\hat{+}\; \Phi_{\tn\aff}(T_2).$$
    \item[(ii)] If $T_1 T_2$ is pre-closed, then $\Phi_{\tn\aff}(T_1) \Phi_{\tn\aff}(T_2)$ is pre-closed, and $$\Phi_{\tn\aff}(T_1 \;\hat{\cdot}\; T_2) = \Phi_{\tn\aff}(T_1) \;\hat{\cdot}\; \Phi_{\tn\aff}(T_2).$$
\end{itemize}
In other words, $\Phi_{\tn\aff}$ preserves strong-sum and strong-product.
}
\end{cor}
\begin{proof}
Since $T_1, T_2, \in \afm$, from Theorem \ref{thm:phi_monoid_morph} we have,
$$\Phi_{\tn\aff}(T_1 + T_2) = \Phi_{\tn\aff}(T_1) + \Phi_{\tn\aff}(T_2), \Phi_{\tn\aff}(T_1 T_2) = \Phi_{\tn\aff}(T_1) \Phi_{\tn\aff}(T_2).$$
Then the assertion is an immediate consequence of Theorem \ref{thm:phi_pres_closed}-(iv).
\end{proof}

A geometric way to build an extension of $\Phi$ to $\afm^{\text{MvN}}$ is via Stone's idea of characteristic projection matrices (see \cite{stone51}). Note that there is a one-to-one correspondence between a closed operator on $\sH$ and its characteristic projection matrix acting on $\sH \oplus \sH$. In Lemma \ref{lem:char_aff}, we have shown that if $T \in \afm^{\text{MvN}}$, then its characteristic projection, $\rchi(T)$, lies in $M_2(\kM)$. The mapping $\Phi$ induces a mapping via $\Phi_{(2)}$ between $M_2(\kM)$ to $M_2(\kN)$ and clearly, sends projection matrices to projection matrices. In the theorem below, we show that this strategy for obtaining an extension of $\Phi$ to $\afm^{\text{MvN}}$ also begets our now-familiar extension $\Phi_{\tn\aff}$.  

\begin{thm}
\label{thm:phi_aff_geom}
\textsl{
Let $(\mathscr{M}; \sH), (\mathscr{N}; \sK)$ be represented von Neumann algebras and $\Phi : \mathscr{M} \to \mathscr{N}$ be a unital normal $*$-homomorphism. Let $\Phi_{\tn\aff} : \afm \to \afn$ be the extension of $\Phi$ as given in Definition \ref{defn:Phi_aff}. Then the following diagram commutes.
}{
\[ 
\begin{tikzcd}
\afm^{\mathrm{MvN}} \arrow{r}{\rchi } \arrow[swap]{d}{\Phi_{\tn\aff}} & M_2(\kM) \arrow{d}{\Phi_{(2)}}\\
\afn^{\mathrm{MvN}} \arrow{r}{\rchi }& M_2(\kN)
\end{tikzcd}
\]
}
\textsl{
In other words, the extension $\Phi_{\tn\aff} : \afm^{\text{MvN}} \to \afn^{\text{MvN}}$ of $\Phi$, may be defined in a geometric manner via characteristic projections as the mapping, 
\[
T \mapsto (\rchi ^{-1} \circ \Phi_{(2)} \circ \rchi ) (T), \text{ for } T \in \afm^{\text{MvN}}.
\]
}
\end{thm}
\begin{proof}
Recall the notation $\fS_{A', B'}$ from \S \ref{sec:prelim}. Let $T \in \afm^{\text{MvN}}$ and $A/B$ be a quotient representation of $T$ for $A, B \in \kM$ with $\nul (B) \subseteq \nul (A)$. By Proposition \ref{prop:phi_aff_is_well_def}, $\Phi(A) / \Phi(B)$ is a quotient representation of $\Phi_{\tn\aff}(T)$. Since $T$ is a closed operator, Theorem \ref{thm:phi_pres_closed}-(iii) tells us that $\Phi_{\tn\aff}(T)$ must be a closed operator, and therefore its graph, $\graph{\Phi_{\tn\aff}(T)}$, is a closed subspace of $\sK \oplus \sK$. Thus,
\begin{align*}
\rchi \big( \Phi_{\tn\aff}(T) \big) &= \rchi \big( \Phi(A) / \Phi(B) \big) \overset{(\ref{eqn:graph_formula})}{=} \raR(\fS_{\Phi(A), \Phi(B)}) \\
&= \raR \left( \Phi_{(2)}(\fS_{A, B}) \right) = \Phi_{(2)}\Big( \raR(\fS_{A, B}) \Big), \text{ by Lemma \ref{lem:morph_prop}}\\
&\overset{(\ref{eqn:graph_formula})}{=} \Phi_{(2)} \big( \rchi(A/B) \big) = \Phi_{(2)} \big( \rchi(T) \big). \qedhere
\end{align*}
\end{proof}

Traditionally, $\afm ^{\text{MvN}}$ is viewed as being equipped with the binary operations of strong-sum and strong-product (which are not necessarily defined for every pair of operators in $\afm^{\text{MvN}}$), and the unary operation of adjoint. In the theorem below, we show that there is a unique extension of $\Phi$ that respects these algebraic operations; yet again, this extension turns out to be the restriction of $\Phi_{\tn\aff}$ (see Definition \ref{defn:Phi_aff}) to $\afm^{\text{MvN}}$. At this point, we hope that we have been able to persuade the reader through various perspectives that $\Phi_{\tn\aff}$ indeed gives the canonical extension of $\Phi$ to $\afm^{\text{MvN}}$.

\begin{thm}
\label{thm:uniqueness_phi_aff_2}
\textsl{
Let $(\mathscr{M}; \sH), (\mathscr{N}; \sK)$ be represented von Neumann algebras and $\Phi : \mathscr{M} \to \mathscr{N}$ be a unital normal $*$-homomorphism. Let $\Phi_{\tn\aff} : \afm \to \afn$ be the extension of $\Phi$ as given in Definition \ref{defn:Phi_aff}. Then the restriction of $\Phi_{\tn\aff}$ to $\afm^{\text{MvN}}$ is the unique mapping from $\afm^{\text{MvN}}$ to $\afn^{\text{MvN}}$ which extends $\Phi$, and respects strong-sum, strong-product and adjoint.
}
\end{thm}
\begin{proof}
Let $\Psi$ be a map from $\afm^{\text{MvN}}$ to $\afn^{\text{MvN}}$ that extends $\Phi$, respects strong-sum, strong-product and adjoint. Before we proceed with the proof, here are a couple of things to keep in mind. From our hypothesis that $\Psi$ maps into $\afn^{\text{MvN}}$, $\Psi(T)$ is a closed operator on $\sK$ for all $T \in \afm^{\text{MvN}}$. If $T_1, T_2$ are closed operators on a Hilbert space such that $T_1 T_2$ is closed, then by definition of the strong-product, $T_1 \hat{\cdot} T_2 = T_1 T_2$. 
\vskip 0.04in

\noindent\textbf{Observation 1.} For every $B \in \kM$, we have $\Psi(B)^\dagger \subseteq \Psi(B^\dagger)$.
\vskip 0.04in
{\noindent\it Proof of Observation 1.} From equation (\ref{eqn:char_of_Kaufman_inv}) note that $B^{\dagger} B = I - \nulN(B)$ is closed. Since $\Psi(B^\dagger) \in \afn^{\text{MvN}}$ is closed and $\Psi(B)=\Phi(B)\in \kN$, by Lemma \ref{lem:unbdd} we observe that $\Psi(B^\dagger) \Psi(B)$ is closed. Thus from the hypothesis on $\Psi$ and Lemma \ref{lem:morph_prop}, we get,
\[
\Psi(B^\dagger)\Psi(B) = \Psi(B^\dagger)\hat{\cdot}\Psi(B) = \Psi(B^{\dagger} \hat{\cdot} B) = \Psi \big( I- \nulN (B) \big) = I - \nulN \big( \Psi(B) \big).
\]
From Remark \ref{rem:kauf_inv_2} and the preceding equation, it follows that $\Psi(B)^\dagger \subseteq \Psi(B^\dagger)$.
\vspace{0.4cm}

\noindent\textbf{Observation 2.} For every $T \in \afm^{\text{MvN}}$, we have $\Phi_{\tn\aff}(T) \subseteq \Psi(T)$.
\vskip 0.04in

{\noindent\it Proof of Observation 2.} Let $A, B \in \kM$ such that $T = AB^\dagger$. From the hypothesis on $\Psi$, we have
\begin{align*}
\label{eqn:psi_extends_phi}
    \Psi(T) 
    = \Psi(A) \hat{\cdot} \Psi(B^{\dagger}) 
    &\supseteq \Psi(A) \hat{\cdot}  \Psi(B)^{\dagger}, \; \text{ by Observation 1 and Remark \ref{rem:prod_ext}} \\
    &= \Phi(A) \hat{\cdot}  \Phi(B)^{\dagger} = \Phi_{\tn\aff}(AB^{\dagger}), \; \text{ by Corollary \ref{cor:phi_pres_st_sum_pdt}-(ii)}\\
    &= \Phi_{\tn\aff}(T).
\end{align*}
\vskip 0.02in

\noindent {\bf Observation 3.} For every projection $E \in \kM$, we have $\Psi(E^{\dagger}) = \Psi(E)^{\dagger}$.
\vskip 0.04in

{\noindent\it Proof of Observation 3.} 
For a closed operator $T$ on $\sH$ and a bounded operator $B$  on $\sH$, note that $T+B$ is closed so that $T \hat{+} B = T+B$.  It is easy to verify that $E^\dagger + (I - E) = E^\dagger$. Applying $\Psi$ on both sides, we get 
\begin{align*}
\Psi(E^\dagger) &= \Psi \left( E^{\dagger} + (I-E) \right) =  \Psi \left( E^{\dagger} \hat{+} (I-E) \right) \\
&= \Psi(E^\dagger) \hat{+} \Psi(I-E) =\Psi(E^{\dagger}) \hat{+} \left(I - \Psi(E)\right)\\
&= \Psi(E^{\dagger}) + \left(I - \Psi(E)\right).
\end{align*}
Note that $\Psi(E) = \Phi(E)$ is a projection in $\kN$. Thus, we have $$\dom \left( \Psi(E^{\dagger}) \right) \subseteq \nul \left( I-\Psi(E) \right) =  \ran \left( \Psi(E) \right) = \dom \left( \Psi(E)^{\dagger} \right).$$
Together with Observation 1, this tells us that $\Psi(E^{\dagger}) = \Psi(E)^{\dagger}$.
\vspace{0.4cm}

\noindent\textbf{Observation 4.} For every $T \in \afm^{\text{MvN}}$, we have $\Psi(T) = \Phi_{\tn\aff}(T)$.   
\vskip 0.04in

{\noindent\it Proof of Observation 4.} Let $S \in \afm^{\text{MvN}}$ be densely-defined.  From Theorem \ref{thm:phi_pres_closed}-(i), we see that $\Phi_{\tn\aff}(S)$ is densely-defined on $\sK$. Since by Observation 2, $\Phi_{\tn\aff}(S) \subseteq \Psi(S)$, clearly $\Psi(S)$ is also densely-defined on $\sK$. Recall that $\Phi_{\tn\aff}$ preserves adjoint (by Theorem \ref{thm:phi_monoid_morph}-(iv)) and by the  hypothesis, $\Psi$ also preserves adjoint. Using \cite[Remark 2.7.5]{kr-I} and Observation 2, we get
\begin{align*}
    \Psi(S)^* \subseteq \Phi_{\tn\aff}(S)^* = \Phi_{\tn\aff}(S^*) \subseteq\Psi(S^*) = \Psi(S)^*.
\end{align*}
From the sandwich of extensions, we have
\begin{equation}
\label{eqn:dd_psi_phi}
\Phi_{\tn\aff}(S^*) = \Psi(S^*),
\end{equation} for all densely-defined $S$ in $\afm^{\text{MvN}}$. 

Let $T\in \afm^{\text{MvN}}$. By Proposition \ref{prop:t_bar}, $T = T^{**}E^\dagger$ for some projection $E\in \kM$. Since $T^{*}$ is densely-defined and $\Psi(E^\dagger)= \Psi(E)^\dagger$ (by Observation 3) the result follows from equation (\ref{eqn:dd_psi_phi}) for $S = T^*$. Indeed,
\[
    \Psi(T) = \Psi(T^{**})\Psi(E)^\dagger = \Phi_{\tn\aff}(T^{**}) \Phi(E)^\dagger = \Phi_{\tn\aff}(T^{**}E^\dagger) = \Phi_{\tn\aff}(T). \qedhere
\]
\end{proof}

There are several special classes of affiliated operators studied in the literature. In fact, non-commutative function theory primarily deals with non-commutative analogues of classical function spaces (such as $L^p$-spaces, Sobolev spaces, etc.)\ which are often related to the set of affiliated operators for a von Neumann algebra. In Remark \ref{rem:mes_op}, we consider the $*$-algebra of measurable operators and the $*$-algebra of locally measurable operators associated with a von Neumann algebra, demonstrating that, unlike the near-semiring of affiliated operators, they are {\bf not} functorial constructions.

\begin{definition}[Measurable Operator, {\cite{segal_integration, yeadon_1973}}]
    \label{def:meas_op}
	An operator $A \in \affc{\kM}$ is called \emph{measurable} iff $\dom(A)$ is dense and $I-E_\lambda$ is \emph{finite} (relative to $\kM$) for some $\lambda >0$ where $|A|=\int_{0}^{\infty}\lambda \mathrm{d}E_\lambda$ is the spectral resolution of $|A|$. The set of all measurable operators is denoted by $\sS(\kM)$.
\end{definition}

\begin{definition}[Locally Measurable Operator, {\cite{yeadon_1973}}]
    \label{def:loc_mes_op}
	An operator $A \in \affc{\kM}$ is said to be \emph{locally measurable} if there exist projections $Q_n\in\kC$, the center of $\kM$, and $E_n\in \kM$ such that $Q_n\uparrow I, E_n\uparrow I, \ran(E_n)\subseteq \dom(A)$ and $Q_n(I-E_n)$ is finite for all $n$. The set of all locally measurable operators associated with $\kM$ is denoted by $\sL(\kM)$. 
\end{definition}

\begin{remark}
\label{rem:mes_op}
Suppose $\Phi:\kM\to\kN$ is a unital normal $*$-homomorphism between two von-Neumann algebras. Now it is natural to ask whether or not $\Phi_{\tn\aff}$ preserves \emph{measurability} of affiliated operators. In general, it doesn't! Indeed, let $T$ be a densely-defined self-adjoint operator acting on $\sH$. Then $T$ generates an abelian von-Neumann algebras $\kA$ such that $T\in \kA_{\aff}$. Since $\kA$ is finite so $T\in \sS(\kA)$, the set of all measurable operators on $\kA$. Now if we consider $\Phi$ to be the inclusion $\kA\hookrightarrow \Bh$, then $\Phi_{\tn\aff}(T)$ is a densely-defined self-adjoint operator on $\sH$ which is not contained in $\sS \big( \Bh \big) = \Bh$; that is, $\Phi_{\tn\aff}$ doesn't preserve the measurability of operators. Similarly, since $\sL \big( \Bh \big) = \Bh$, $\Phi_{\tn\aff}$ doesn't preserve the local measurability either.
\end{remark}

\section{Some Operator Properties preserved by {$\Phi_{\textnormal{aff}}$}}%
\label{sec:misc_alg_props}%

In this section, we show that $\Phi_{\aff}$ preserves operator properties such as being \emph{symmetric}, \emph{positive}, \emph{self-adjoint}, \emph{normal}, \emph{accretive}, \emph{sectorial}. First we recall some definitions.

\begin{definition}
\label{defn:symm_pos}
An unbounded operator $T$ on the Hilbert space $\sH$ is said to be
\begin{itemize}
    \item[(i) ] \emph{symmetric} if $\ipdt{Tx}{x} \in \R$ for all $x\in \dom(T)$;
    \item[(ii)] \emph{positive} if $\ipdt{Tx}{x}\ge 0$ for all $x \in \dom(T)$;
    \item[(iii)] {\it accretive} if $\Re\ipdt{Tx}{x}\ge 0$ for all $x\in \dom(T)$;
    \item[(iv)] {\it sectorial} if $\ipdt{Tx}{x} \in S_{c,\theta}:=\{\lambda\in \C~:~|\arg(\lambda-c)|\le \theta\} \text{ for all } x\in \dom(T),$
 where $c\in \C$ and $\theta \in [0, \frac{\pi}{2})$.
 \end{itemize}
The \emph{numerical range} of $T$, denoted by $W(T)$, is defined by
	\[
	W(T):=\{\ipdt{Tx}{x}:x\in \mathrm{dom}(T), \|x\|=1\}.
	\]
\end{definition}

\begin{remark}
Let $T$ be an unbounded operator on the Hilbert space $\sH$. We denote the closed right half-plane in $\C$ by,
\[
	\mathbb{H}_{0, 0}  := \{ z \in \C~:~\Re(z)\ge 0 \}.
\] Table \ref{table:numerical_ran} recasts the operator properties described in Definition \ref{defn:symm_pos} in terms of numerical ranges.

\begin{table}[h!]
    \centering
    \begin{tabular}{|c | c|} 
        \hline
        $T$ is symmetric & $W(T) \subseteq \R$  \\[1ex]
        \hline
        
        $T$ is positive & $W(T) \subseteq \R_{\ge 0}$  \\[1ex]
        \hline

        $T$ is accretive & $W(T) \subseteq \mathbb{H}_{0, 0}$  \\[1ex]
        \hline

        $T$ is sectorial & $W(T) \subseteq S_{c, \theta}$  \\[1ex]
        \hline
    \end{tabular}
    \vspace{0.3cm}
    \caption{Operator properties in terms of their numerical ranges}
    \label{table:numerical_ran}
\end{table}

\end{remark}

For a complex number $\alpha\in \C$ and an angle $\theta\in [-\pi, \pi)$ we will denote the affine transformation $\C\to\C:z\mapsto e^{i\theta}(z+\alpha)$ by $\fL_{\theta, \alpha}$. Every \emph{closed half-space} in $\C$ is the image of $\mathbb{H}_{0, 0}$ under some $\fL_{\theta, \alpha}$ and is denoted by $\mathbb{H}_{\theta, \alpha}$, that is, 
\[
	\mathbb{H}_{\theta, \alpha} := \fL_{\theta, \alpha}(\mathbb{H}_{0, 0}).
\]
Note that the affine transformation $\fL_{\theta, \alpha}$ is invertible and its inverse $\fL_{\theta, \alpha}^{-1} = \fL_{-\theta, -e^{i \theta} \alpha}$ is an affine transformation. For the affine transformation, $\fL_{\theta, \alpha}$, we define $\fL_{\theta, \alpha}(T) := e^{i\theta}(T + \alpha I)$. Clearly, $\dom \big( \fL_{\theta, \alpha}(T) \big) = \dom(T).$

\begin{lem}
\textsl{
For an operator $T$ acting on the Hilbert space $\sH$ we have,
	\[
		W \big( \fL_{\theta, \alpha}(T) \big) = \fL_{\theta, \alpha} \big( W(T) \big),
	\]
	where $\alpha\in \C$ and $\theta\in [0, \frac{\pi}{2})$.
}
\end{lem}
\begin{proof}
Just unwrapping the appropriate definitions, we get 
\[
W \big( \fL_{\theta, \alpha}(T) \big) = W \left( e^{i\theta}(T+\alpha I) \right) = e^{i\theta}(W(T) + \alpha)=\fL_{\theta, \alpha} \big( W(T) \big).   \qedhere
\]
\end{proof}

\begin{prop}
\label{prop:phi_num_range}
\textsl{
Let $(\mathscr{M}; \sH), (\mathscr{N}; \sK)$ be represented von Neumann algebras and $\Phi : \mathscr{M} \to \mathscr{N}$ be a unital normal $*$-homomorphism. For $T \in \afm$, if $W (T) \subseteq \mathbb{H}_{0, 0}$, then $W \big( \Phi_{\aff}(T) \big) \subseteq \mathbb{H}_{0, 0}$. 
}
\end{prop}
\begin{proof}
Let $A', B'$ be bounded operators on a Hilbert space with $\nul (B') \subseteq \nul (A')$, and let $T' = A'/B'$. Keeping in mind that $\dom (T') = \ran(B')$, we have,
\begin{align*}
W (T') \subseteq \mathbb{H}_{0, 0} 
& \iff \forall y\in \dom(T'), 2 \Re \ipdt{T'y}{y} =  \ipdt{T'y}{y} + \ipdt{y}{T'y}\ge 0\\
& \iff \forall x \in \sH, \ipdt{T' B'x}{B'x} + \ipdt{B'x}{T'B'x} \ge 0\\
& \iff \forall x \in \sH, \ipdt{A'x}{B'x} + \ipdt{B'x}{A'x} \ge 0\\
& \iff B'^*A' + A'^*B' \ge 0.
\end{align*}

Let $A/B$ be a quotient representation of $T$ so that $\Phi(A)/\Phi(B)$ is a quotient representation of $\Phi_{\aff}(T)$. From the above discussion, we see that $B^*A + A^*B \ge 0$ and by positivity of $\Phi$, clearly $\Phi(B)^*\Phi(A) + \Phi(A)^*\Phi(B) \ge 0$, whence $W (\Phi_{\aff}(T)) \subseteq \mathbb{H}_{0, 0}$.
\end{proof}

\begin{prop}
\label{prop:phi_num_range2}
\textsl{
Let $(\mathscr{M}; \sH), (\mathscr{N}; \sK)$ be represented von Neumann algebras and $\Phi : \mathscr{M} \to \mathscr{N}$ be a unital normal $*$-homomorphism.\ Let $T \in \afm$ and $K$ be a closed convex set in $\C$. If $W (T) \subseteq K$, then $W \big( \Phi_{\aff}(T) \big) \subseteq K$.
}
\end{prop}
\begin{proof}
If $K = \C$, there is nothing to prove. Thus we may assume that $K$ is a proper subset of $\C$. First, we prove the assertion for the closed half-plane $K=\mathbb{H}_{\theta, \alpha}$ where $\theta \in [-\pi, \pi)$ and $\alpha \in \C$. For $\theta ' = -\theta, \alpha' = -e^{i \theta} \alpha$, note that  $\fL_{\theta ',\alpha '} = \fL_{\theta, \alpha}^{-1}$. We have, \begin{align*}
W(T)\subseteq \mathbb{H}_{\theta, \alpha} & \implies W \left( \fL_{\theta ', \alpha '}(T) \right) = \fL_{\theta ', \alpha '} \big( W(T) \big)\subseteq \fL_{\theta ', \alpha '} (H_{\theta, \alpha}) = \mathbb{H}_{0, 0}\\
&\implies W \big( \Phi_{\aff}(\fL_{\theta ', \alpha '}(T)) \big) \subseteq \mathbb{H}_{0, 0}, \;\;\text{by Proposition \ref{prop:phi_num_range}}\\
& \implies W \Big( \fL_{\theta ', \alpha '} (\Phi_{\aff}(T) \big) \Big) = \fL_{\theta ', \alpha '} \Big( W \big( \Phi_{\aff}(T) \big) \Big) \subseteq \mathbb{H}_{0, 0}\\
& \implies W \big( \Phi_{\aff}(T) \big) \subseteq \fL_{\theta, \alpha} (\mathbb{H}_{0, 0}) = \mathbb{H}_{\theta, \alpha}.
\end{align*}

From here, the assertion follows from a basic theorem in convex geometry which asserts that every closed convex proper subset of $\C$ (in fact, for any Euclidean space) is an intersection of closed half-spaces (see \cite[Corollary 1.3.5]{schneider_convex}). 
\end{proof}

\begin{thm}
\label{thm:sym_acc_sect}
\textsl{
Let $(\mathscr{M}; \sH), (\mathscr{N}; \sK)$ be represented von Neumann algebras and $\Phi : \mathscr{M} \to \mathscr{N}$ be a unital normal $*$-homomorphism. Let $\Phi_{\aff} : \afm \to \afn$ be the extension of $\Phi$ as given in Definition \ref{defn:Phi_aff}, and let $T \in \afm$.
\begin{itemize}
    \item[(i)] If $T$ is symmetric, then so is $\Phi_{\aff}(T)$.
    \item[(ii)] If $T$ is positive, then so is $\Phi_{\aff}(T)$.
    \item[(iii)] If $T$ is accretive, then so is $\Phi_{\aff}(T)$.
    \item[(iv)] If $T$ is sectorial, then so is $\Phi_{\aff}(T)$.
    \item[(v)] If $T$ is self-adjoint, then so is $\Phi_{\aff}(T)$.
    \item[(vi)] If $T$ is normal, then so is $\Phi_{\aff}(T)$.
\end{itemize}
}
\end{thm}
\begin{proof}
Since $\R, \R_{\ge 0}, \mathbb{H}_{0, 0}, S_{c, \theta}$ are closed convex sets in $\C$, the assertions in parts (i)-(iv) follow from Proposition \ref{prop:phi_num_range2} and the description of these operators in Table \ref{table:numerical_ran}. Parts (v) and (vi) are immediate from Theorem \ref{thm:phi_monoid_morph}.
\end{proof}

\section{The Krein Extension Theory for Positive Affiliated Operators}%
\label{sec:krein_friedrichs}%

Consider the following classical question: Given a \emph{positive (symmetric)} operator $S$ acting on the Hilbert space $\sH$, does there always exist a \emph{positive self-adjoint} extension of $S$? Since self-adjoint operators are closed, for an affirmative answer, $S$ must be pre-closed, and without loss of generality, we may assume that $S$ is closed. When $S$ is densely-defined, the answer to the above question is affirmative; two such classical examples are given by the Friedrichs extension (see \cite{friedrichs_1934}) which we denote by $\fried{S}$, and the Krein-von Neumann extension (see \cite{krein_1947}) which we denote by $\krein{S}$. In fact, among all possible such extensions, $\fried{S}$ is the \emph{maximal} one, and $\krein{S}$ is the {\it minimal} one. 

Throughout this section, $\kM$ denotes a von Neumann algebra acting on the Hilbert space $\sH$. Let $S$ be a densely-defined closed positive operator affiliated with $\kM$. The Friedrichs and Krein-von Neumann extensions, $\fried{S}, \krein{S}$, are both affiliated with $\kM$ (proofs can be found in \cite[Appendix]{kadison_dual_cone}, \cite[\S 4]{skau_paper}, respectively). Let $\mathscr{N}$ be a von Neumann algebra acting on the Hilbert space $\sK$ and $\Phi : \mathscr{M} \to \mathscr{N}$ be a unital normal $*$-homomorphism. By Theorem \ref{thm:phi_pres_closed}-(i) and Theorem \ref{thm:sym_acc_sect}, we note that $\Phi_{\tn\aff}(S)$ is a densely-defined closed positive operator affiliated with $\kN$. From Theorem \ref{thm:sym_acc_sect} and Theorem \ref{thm:phi_pres_ext}, it is clear that $\Phi_{\tn\aff}(\fried{S}), \Phi_{\tn\aff}(\krein{S})$ are positive self-adjoint extensions of $\Phi_{\tn\aff}(S)$ which are affiliated with $\kN$. In this section, our primary objective is to show that these are in fact respectively the Friedrichs and Krein extensions of $\Phi_{\tn\aff}(S)$, that is,
\[
	\Phi_{\tn\aff}(\fried{S})=\fried{\Phi_{\tn\aff}(S)}\text{ and }~\Phi_{\tn\aff}(\krein{S})=\krein{\Phi_{\tn\aff}(S)}.
\]
To this end, we need to discuss Krein's extension theory (see \cite{krein_1947}) for a positive operator on a Hilbert space. A rough sketch may be found in \cite{skau_paper} and a more comprehensive account can be found in \cite[Chapter 13]{konrad_unbdd}. Since Krein's extension theory has an algebraic flavour, we expect it to be compatible with morphisms in $\kW^*$-Alg. Below we give a reasonably self-contained account rephrasing the results from \cite[\S 4]{skau_paper} in a language which makes functoriality evident.

\begin{definition}
Let $\sC (\sH)^{+}$ denote the family of closed positive operators on $\sH$ and $\sF(\sH)$ denote the family of closed and bounded (not necessarily everywhere-defined) symmetric operators $B$ with $||B||\le 1$ and $\nulN(I-B)= 0 $.

For an operator $S \in \sC (\sH)^{+}$, note that $\nulN (S+I) = 0$ so that $S+I$ is one-to-one. The \emph{Krein transform} of $S$ is defined as,
\[
	\fK(S):=(S-I)(S+I)^{-1}; ~\dom\left(\fK(S)\right) = \ran(S+I).
\]
For $B\in \sF(\sH)$, the \emph{inverse Krein transform} of $B$ is defined as,
\[
	\fK^{-1}(B):=(I+B)(I-B)^{-1}; ~\dom \left( \fK^{-1}(B) \right) = \ran(I-B).
\]
Note that, 
\begin{equation}
    \label{eqn:krein_transform}
    \begin{aligned}
        \fK(S) &= (S-I)(S+I)^{-1} = (S+I - 2I)(S+I)^{-1} = I - 2(S+I)^{-1},\\
    \fK^{-1}(B) &= (I+B)(I-B)^{-1} = \big( 2I - (I-B) \big) (I-B)^{-1} = 2(I-B)^{-1} - I.
    \end{aligned}
\end{equation}

\end{definition}

\begin{prop}
\label{prop:kr_SplusI}
\textsl{
Let $S$ be a closed positive operator on $\sH$. Then we have the following.
\begin{itemize}
    \item[(i)] $\nul (S+I) = \{ 0 \}_{\sH}$, and $\ran(S+I)$ is a closed subspace of $\sH$. If $\nul (S) = \{ 0 \}_{\sH}$, then $S^{-1}$ is a closed positive operator on $\sH$.
    \item[(ii)] $(S+I)^{-1}$ is a closed bounded positive operator with norm less than or equal to $1$, and has trivial nullspace. $\big($Note that $\dom \left( (S+I)^{-1} \right) = \ran (S+I)$.$\big)$
    \item[(iii)] $S$ is self-adjoint if and only if $\ran(S+I) = \sH$;
    \item[(iv)] $S_0$ is a positive extension of $S$ if and only if $(S_0+I)^{-1}$ is a positive extension of $(S+I)^{-1}$.
\end{itemize}
}
\end{prop}
\begin{proof}
Parts (i) and (ii) follows from \cite[Lemma 2.7.9]{kr-I} and (iii) follows from \cite[Proposition 2.7.10]{kr-I}. For (iv), note that $\ran(S+I) \subseteq \ran(S_0+I)$ or equivalently, $\dom \big( (S+I)^{-1} \big) \subseteq \dom \big( (S_0 +I)^{-1} \big)$, and then use part (ii).
\end{proof}

Using Proposition \ref{prop:kr_SplusI}, we have the following result.

\begin{prop}[cf.\ {\cite[Proposition 13.22]{konrad_unbdd}}]
\label{prop:krein_tranform}
\textsl{
The mapping $\fK$ gives a bijection between $\sC(\sH)^{+}$, and $\sF(\sH)$ and its inverse mapping is the inverse Krein transform, $\fK^{-1}$. Furthermore, $S_0$ is a positive self-adjoint extension of $S$ if and only if $\fK(S_0)$ is a positive self-adjoint extension of $\fK(S)$.
}
\end{prop}

\begin{lem}
\label{lem:fundamental}
\textsl{
Let $\kM$ be a von-Neumann algebra acting on the Hilbert space $\sH$ and $A\in \kM^+, E\in\kM_{\text{proj}}$ be given. Then the set 
	\[
		\{D \in \sB(\sH)_{sa}: \raR(D) \le E, D\le A\}
	\]
	contains a largest operator $\sD(A; E)$, which lies in $\kM$ and is given by the following expression,
\begin{equation}
\label{eqn:d_a_e}
	\sD(A; E) = A^{\frac{1}{2}} \nulN \left((I-E) A^{\frac{1}{2}}\right)
    A^{\frac{1}{2}}.   
\end{equation}
}
\end{lem}

\begin{proof}
Note that,
\begin{align*}
x \in \nul ((I-E)A^{\frac{1}{2}}) &\iff (I-E) A^{\frac{1}{2}} x = 0 \iff  E(A^{\frac{1}{2}} x) = A^{\frac{1}{2}} x\\
&\iff A^{\frac{1}{2}} x \in \ran(E).
\end{align*}
Thus $\nulN \left((I-E) A^{\frac{1}{2}}\right)$ is the projection onto the set  $\{ x \in \sH : A^{\frac{1}{2}}x \in \ran(E) \}$, and clearly it lies in $\kM$ as $(I-E)A^{\frac{1}{2}} \in \kM$. With this observation in mind, we just have to follow the line of argument in the (elementary) proof of the ``Fundamental Lemma'' in \cite[pg. 179]{skau_paper} to reach our desired conclusion.
\end{proof}

Let $S$ be a closed positive (symmetric) operator affiliated with $\kM$. When $S$ is densely-defined, a positive self-adjoint extension of $S$ always exists (\cite[Theorem 10.17]{konrad_unbdd}). But in general, this may not be true (see \cite[Example 13.2, Exercise 13.6.9(b)]{konrad_unbdd}).\ Nevertheless, from Proposition \ref{prop:krein_tranform} and Lemma \ref{lem:phi_krein}, it follows that there is a one-to-one correspondence between the set of positive self-adjoint extensions of $S$ which are affiliated with $\kM$, and the set of self-adjoint extensions of the Krein transform of $S$, $\fK(S)$, which lie in $\kM \cap \sF(\sH)$. Below we see how to find all such self-adjoint extensions of $\fK(S)$ using Lemma \ref{lem:fundamental}, starting with a particular example of such an extension.

\begin{prop}
\label{prop:kt_min_max}
\textsl{
Let $\kM$ be a von Neumann algebra acting on the Hilbert space $\sH$ and $S$ be a positive (symmetric) operator in $\afm$. Let $K$ be a self-adjoint extension of the Krein transform of $S$, $\fK(S)$, lying in $\kM \cap \sF(\sH)$, and $F$ be the projection onto $\dom(\fK(S))^\perp$. Let $\sD(\cdot\; ; \cdot)$ be as in Lemma \ref{lem:fundamental}.
\begin{itemize}
    \item[(i)] Then $K - \sD(I+K; F)$ is a self-adjoint extension of $\fK(S)$ lying in $\kM \cap \sF(\sH)$. For every self-adjoint extension $K'$ of $\fK(S)$, we have 
    \[
        K - \sD(I+K;F) \le K'.
    \] 
    In particular, for any two self-adjoint extensions $K, K'$ of $\fK(S)$ lying in $\kM \cap \sF(\sH)$, we have 
\begin{equation}
\label{eqn:K_m}
        K - \sD(I+K;F) = K' - \sD(I+K';F) (=: \fK(S)_{\min}).
\end{equation}
    \item[(ii)] In addition, if $S$ is densely-defined, then $K + \sD(I-K; F)$ is a self-adjoint extension of $\fK(S)$ lying in $\kM \cap \sF(\sH)$. For any other self-adjoint extension $K'$ of $\fK(S)$, we have 
    \[
        K' \le K + \sD(I-K;F).
    \]  
    Thus, for any two self-adjoint extensions $K, K'$ of $\fK(S)$ lying in $\kM \cap \sF(\sH)$, we have 
\begin{equation}
\label{eqn:K_M}
            K + \sD(I-K;F) = K' + \sD(I-K';F) (=: \fK(S)_{\max}).
\end{equation}
\end{itemize}
}
\end{prop}
\begin{proof}
Since $\dom (\fK(T)) \in \affs{\kM}$, clearly $F \in \kM$. For the sake of notational simplicity, we define, 
\[
K_{m} := K - \sD(I+K;F),\; K_{M} := K + \sD(I-K;F).
\]

Since $K \in \kM \cap \sF(\sH)$, we have $\|K\| \le 1$ so that $-I \le K \le I$. From Lemma \ref{lem:fundamental}, we know that 
\begin{align}
\sD(I+K; F) &\in \kM^{+}, &\sD (I-K; F) &\in \kM^{+}, \label{eqn:sa_1}\\
\sD(I+K; F) &\le I+K, &\sD(I-K; F) &\le I-K,\label{eqn:sa_2}\\
\raR \big( \sD(I + K; F) \big) &\le F, &\raR \big( \sD(I - K; F) \big) &\le F.\label{eqn:sa_3}
\end{align}

Let $K'$ be a self-adjoint extension of $\fK(T)$ lying in $\kM \cap \sF(\sH)$ and $C:=K'-K$, that is, $K' = K + C$. Then it is easy to see that $K'\in \sF(\sH)$ if and only if $C$ is self-adjoint, $\raR(C) \le F$ and $-(I+K) \le C \le I - K$. Hence applying Lemma \ref{lem:fundamental}, $C$ satisfies
\[
-\sD(I+K, F) \le C \le \sD(I-K, F).
\]
Adding $K$ to all terms of the inequality, we have 
\begin{equation}
\label{eqn:K_m_le_K_K_M}
K_m \le K' \le K_M.    
\end{equation}

A key thing to verify is whether $K_m, K_M$ lie in $\sF(\sH)$.
\vskip 0.1in

\noindent {\bf Observation 1:} $K_m, K_M$ are self-adjoint extensions of $\fK(S)$ lying in $\kM$.
\vskip 0.02in
\noindent {\it Proof of Observation 1}: Note that (\ref{eqn:sa_3}) is equivalent to,
$$I-F \le \nulN \big( \sD(I + K; F) \big),\;\; I-F \le \nulN \big( \sD(I - K; F) \big).$$
Since $\dom \big( \fK(S) \big) \subseteq \ran(I-F)$, for every $x\in \dom \big( \fK(S) \big)$ we get, $$\sD(I + K; F)x = \sD(I - K; F)x = 0,$$ whence $K_m = K_M = K$ on $\dom \big( \fK(S) \big)$. 
\vskip 0.1in

\noindent {\bf Observation 2: } $\|K_m \| \le 1, \|K_M\| \le 1$.
\vskip 0.02in
\noindent {\it Proof of Observation 2}: From (\ref{eqn:sa_1}) and (\ref{eqn:sa_2}), clearly
\[
-I \le K_m \le K \le K_M \le I.
\]

\vskip 0.1in
\begin{itemize}[leftmargin=0.8cm, itemsep=0.2cm]
    \item[(i)] Since $0 \le I - K \le I-K_m$, we have $\nul (I-K_m) \subseteq \nul (I-K) = \{ 0 \}_{\sH}$ so that $\nul (I-K_m) = \{ 0 \}_{\sH}$. Together with Observations 1 and 2, this tells us that $K_m \in \kM \cap \sF(\sH)$. In other words, for every self-adjoint extension $K$ of $\fK(S)$ lying in the $\kM \cap \sF(\sH)$, the operator $K_m$ is again such an extension of $\fK(T)$. From equation (\ref{eqn:K_m_le_K_K_M}), we have $$K_m\le K'.$$ 
    In particular, the above inequality is also valid for  $(K')_m$ so that $K_m \le (K')_m$ for every pair of self-adjoint extensions, $K, K'$, of $\fK(S)$ lying in $\kM \cap \sF(\sH)$. By symmetry, it is clear that 
    $$K_m = (K')_m.$$

    \item[(ii)] If $S$ is densely-defined, then $I-\fK(S)$ has dense range. Thus for any self-adjoint extension $K$ of $\fK(S)$, $I-K$ has dense range or equivalently, trivial nullspace. In conjunction with Observations 1 and 2, we have $K_M \in \kM \cap \sF(\sH)$. By a similar argument as in part (i) above, it follows that $K_M = (K')_M$ for every pair of self-adjoint extensions, $K, K'$, of $\fK(S)$ lying in $\kM \cap \sF(\sH)$.   \qedhere
\end{itemize}
\end{proof}

\begin{remark}
\label{rem:kr_fr_pres}
Let $S$ be a densely-defined positive symmetric operator. Let $K$ be the Krein transform of a positive self-adjoint extension of $S$ (such an extension always exists as discussed above). Note that $K$ is a positive self-adjoint extension of $\fK(S)$, and by Proposition \ref{prop:kt_min_max}, the extensions 
\begin{equation}
\label{eqn:min_max_formula}
    \fK(S)_{\min}:=K-\sD(I+K; F),\;\; \fK(S)_{\max}:=K+\sD(I-K; F),
\end{equation}
are respectively the minimal and maximal self-adjoint extensions of $\fK(T)$ lying inside $\kM\cap \sF(\sH)$. Although already noted in equations (\ref{eqn:K_m}) and (\ref{eqn:K_M}), we take a moment to emphasize that the expressions in (\ref{eqn:min_max_formula}) evaluate to the minimal and maximal extensions, respectively, {\it independent} of the choice of $K$. 

Using the inverse Krein transform and Proposition \ref{prop:krein_tranform}, we conclude that $S$ has a minimal and maximal self-adjoint extension, respectively given by, 
\[
\fK^{-1} \big( \fK(S)_{\min} \big), \;\; \fK^{-1} \big( \fK(S)_{\max} \big).
\]
In fact, as shown in \cite[\S 4]{skau_paper}, these extensions respectively recover the Krein-von Neumann and Friedrichs extensions of $S$, that is,
\begin{equation}
\label{eqn:kr_fr_min_max}
\krein{S} = \fK^{-1}(\fK(S)_{\min}),\;\; \fried{S}=\fK^{-1}(\fK(S)_{\max}). 
\end{equation}
\end{remark}

Let $(\afm^c)^{+}$ denote the set of closed positive operators affiliated with $\kM$. It is now clear that if $S \in (\afm ^c)^{+}$ and $ B \in \afm ^{cb} \cap \sF(\sH)$, then $\fK(S) \in \afm^{cb} \cap \sF(\sK)$ and $\fK^{-1}(B) \in (\afm ^c)^{+}$.

\begin{lem}
\label{lem:phi_krein}
\textsl{
Let $(\mathscr{M}; \sH), (\mathscr{N}; \sK)$ be represented von Neumann algebras and $\Phi : \mathscr{M} \to \mathscr{N}$ be a unital normal $*$-homomorphism. Then $\Phi_{\text{aff}} \circ \fK=\fK\circ \Phi_{\text{aff}}$ on $(\afm^c)^+$, that is, the following diagram commutes:
}
	\[
		\begin{tikzcd}
			(\afm^c)^+ \arrow[r, "\Phi_{\textnormal{aff}}"] \arrow[d, "\mathfrak{K}"'] & (\afn^c)^+ \arrow[d, "\mathfrak{K}"] \\
			\afm^{cb}\arrow[r, "\Phi_{\textnormal{aff}}"']  & \afn^{cb}         
		\end{tikzcd}
	\]
\end{lem}
\begin{proof}
	For $S\in (\afm^c)^+$, from equation (\ref{eqn:krein_transform}) and Theorem \ref{thm:phi_monoid_morph}  we have,
\begin{align*}
    \Phi_{\tn\aff} \big( \fK(S) \big)
    &=\Phi_{\tn\aff}\left(I-2(S+I)^{-1}\right) = I - 2 \Phi_{\tn\aff}\big( (S+I)^{-1} \big)\\
    &= I - 2 \left( \Phi_{\tn\aff}(S)+I\right)^{-1} 
    =\fK \big( \Phi_\aff(S) \big).   \qedhere
\end{align*}
\end{proof}

In the next theorem, we show that $\Phi_\aff$ preserves $\fK(S)_{\min}$ and $\fK(S)_{\max}$ which yields our main result that $\Phi_{\tn\aff}$ maps $\krein{S}, \fried{S} \in \afm^{\text{MvN}}$ to their respective counterparts for $\Phi_{\tn\aff}(S)$.
\begin{thm}
\label{thm:phi_min_max}
\textsl{
Let $(\mathscr{M}; \sH), (\mathscr{N}; \sK)$ be represented von Neumann algebras and $\Phi : \mathscr{M} \to \mathscr{N}$ be a unital normal $*$-homomorphism. Let $S$ be a densely-defined closed positive operator in $\afm$. (Note that by Theorem \ref{thm:phi_pres_closed}-(i) and Theorem \ref{thm:sym_acc_sect}, $\Phi_{\tn\aff}(S)$ is a densely-defined closed positive operator in $\afn$.) Then $\Phi_{\tn\aff}$ maps the Krein-von Neumann extension (the Friedrichs extension, respectively) of $S$ to the Krein-von Neumann extension (the Friedrichs extension, respectively) of $\Phi_{\tn\aff}(S)$, that is, 
	\[
	\Phi_{\tn\aff}(\krein{S}) = \krein{\Phi_{\tn\aff}(S)}\text{ and }~
	\Phi_{\tn\aff}(\fried{S}) = \fried{\Phi_{\tn\aff}(S)}.
	\]
}
\end{thm}
\begin{proof}
Let $\sD(\cdot\; ; \cdot)$ be as in Lemma \ref{lem:fundamental}.
For $A \in \kM^{+}, E \in \kM_{\text{proj}}$, it is clear from Lemma \ref{lem:morph_prop} and equation (\ref{eqn:d_a_e}) that,
\begin{equation}
    \label{eqn:phi_d}
    \Phi\big( \sD(A; E)\big) = \sD \big( \Phi(A); \Phi(E) \big).
\end{equation}
Let $K$ be a bounded self-adjoint extension of $\fK(S)$ lying in $\kM \cap \sF(\sH)$; (The Krein transform of the Friedrichs extension of $S$ is one such extension.)
\vskip 0.05in
\noindent {\bf Observation 1:} $\Phi(K)$ is a bounded self-adjoint extension of $\fK \big( \Phi_{\tn\aff}(S) \big)$ lying in $\kN \cap \sF(\sK)$.
\vskip 0.05in
\noindent {\it Proof of Observation 1:} From Theorem \ref{thm:phi_pres_ext}, we observe that $\Phi(K)$ is a bounded self-adjoint extension of $\Phi \big( \fK(S) \big) = \fK \big( \Phi_{\tn\aff}(S) \big)$. Since $K \in \sF(\sH)$, note that $\|K\| \le 1$ and $\nul(I-K) = \{ 0 \}_{\sH}$. Thus $\|\Phi(K)\| \le 1 $ and using Lemma \ref{lem:morph_prop}, we get $\nul \big( I- \Phi(K) \big) = \{0\}_{\sK}$, so that $\Phi(K) \in \sF(\sK)$.
\vskip 0.05in

\noindent {\bf Observation 2:} $\Phi\big( \fK(S)_{\min} \big) = \Big(\fK \big( \Phi_{\tn\aff}(S) \big) \Big)_{\min}$, $\Phi\big( \fK(S)_{\max} \big)= \Big(\fK\big(\Phi_{\tn\aff}(S)\big)\Big)_{\max}$.
\vskip 0.02in
\noindent {\it Proof of Observation 2:} Using Observation 1 in the context of Remark \ref{rem:kr_fr_pres}, we note that,
\begin{align*}
    \Big(\fK\big(\Phi_{\tn\aff}(S)\big)\Big)_{\min} 
    &\overset{(\ref{eqn:min_max_formula})}{=}
    \Phi(K) - \sD\big( I+\Phi(K); \Phi(F)\big) \\
    &\overset{(\ref{eqn:phi_d})}{=}
    \Phi \big( K - \sD\left( I + K; F \right) \big)\\
    &\overset{(\ref{eqn:min_max_formula})}{=}
    \Phi\big( \fK(S)_{\min} \big). 
\end{align*}
A similar approach gives us, $\Phi\big( \fK(S)_{\max} \big)= \Big(\fK\big(\Phi_{\tn\aff}(S)\big)\Big)_{\max}$.

Thus we have, \begin{align*}\fK\left( \krein{\Phi_{\tn\aff}(S)} \right) &= \Big(\fK \big( \Phi_{\tn\aff}(S) \big) \Big)_{\min}, \text{ by } (\ref{eqn:kr_fr_min_max})\\
&=\Phi\left( \fK(S)_{\min} \right), \text{ by Observation 2}\\
&= \Phi \big( \fK(\krein{S}) \big), \text{ by } (\ref{eqn:kr_fr_min_max})\\
&= \fK \big( \Phi_{\tn\aff}(\krein{S}) \big), \text{ by Lemma } \ref{lem:phi_krein}. 
\end{align*}
An application of the inverse Krein transform on both sides leads to the desired conclusion,
\[
 \krein{\Phi_{\tn\aff}(S)} = \krein{\Phi_{\tn\aff}(S)} .
\]
The proof for $\fried{S}$ is similar.
\end{proof}

\section{An Application of Functoriality}%
\label{sec:applications}%

Experience shows that the domains of unbounded operators can behave in mysterious ways, which is a common source of frustration for the uninitiated. The literature abounds in examples and counterexamples discussing these peculiarities. The literature is equally sparse regarding analogous results in the context of affiliated operators. In this section, we hope to remedy this.

Let $\kM$ be a properly infinite von Neumann algebra acting on the Hilbert space $\sH$. Consider the following three questions.
\vskip 0.05in
\noindent {\it Question 1.} Is there a positive self-adjoint operator $T$ in $\afm^c$ and a unitary operator $U \in \kM$ such that $\dom (T) \cap \dom(UTU^*) = \{ 0 \}_{\sH}$?
\vskip 0.02in

\noindent {\it Question 2.} Is there a densely-defined closed symmetric operator $A$ in $\afm ^c$ such that $\dom (A^2) = \{ 0 \}_{\sH}$? 
\vskip 0.02in

\noindent {\it Question 3.}\ Are there non-singular positive self-adjoint operators $A, B \in \afm ^c$ satisfying $\dom(A) \cap \dom(B) = \dom (A^{-1}) \cap \dom (B^{-1}) = \{ 0 \}_{\sH}$?
\vskip 0.05in

Since the commutant of $\sB(\sH)$ is trivial, recall that every closed operator on $\sH$ is affiliated with $\sB(\sH)$, that is, $\sC(\sH) = \sB(\sH)_{\aff}^c$. Specialists in the theory of unbounded operators will immediately recognize the inspiration behind the above-mentioned questions. In the context of $\kM = \sB(\sH)$, the first one is a weaker version of a classical result due to von Neumann (see \cite{neumann1929}); The second problem has been studied by Neumark (see \cite{neumark_square}) with stronger versions explored by Schm\"{u}dgen (see \cite{konrad_1983}) and an explicit example also given by Chernoff (see \cite{chernoff_paul}); The third problem is studied by Kosaki in \cite{kosaki_domains}. Using the functoriality of the near-semiring of affiliated operators, we provide a strategy for answering questions of a similar nature that may occur during day-to-day musings of the working operator algebraist. The main ingredient is the following well-known proposition.

\begin{prop}[see {\cite[Proposition V.1.22]{takesaki-I}}]
\label{prop:bh_embed}
\textsl{
Let $\kM$ be a properly infinite von Neumann algebra, and $\sH$ be a separable Hilbert space. Then there is a unital normal embedding of $\sB(\sH)$ into $\kM$. 
}
\end{prop}
\begin{proof}
Without loss of generality, we may assume that $\sH = \ell^2(\N)$. Using the so-called halving lemma for properly infinite projections (\cite[Lemma 6.3.3]{kr-II}), it can be shown that that there exists a sequence of mutually orthogonal projections $\{E_n\}_{n=1}^{\infty}$ such that $E_n \sim E_m$ for all $n, m$ and $\sum_{n} E_n = I$; the details may be found in \cite[Theorem 6.3.4]{kr-II}. This yields a countable system of matrix units in $\kM$, and the assertion follows.
\end{proof}

Proposition \ref{prop:bh_embed} helps us in viewing properly infinite von Neumann algebras as a larger universe containing $\sB \big( \sH \big)$. Since a unital normal embedding of $\sB \big( \sH \big)$ into $\kM$ lifts to a canonical embedding of $\sC(\sH)$ into $\afm$, note that $\afm$ provides a larger mathematical universe where one can do physics. For example, pseudo-differential operators in $\sC \big( H^k(\R) \big)$ (with $H^k(\R)$ a Sobolev space) and differential equations may be transported to the setting of $\afm$; the Heisenberg commutation relation can find a home in $\afm$, that is, there are closed operators $P, Q \in \afm^{\text{MvN}}$ such that $PQ-QP$ is pre-closed and densely-defined with closure $I$. Although what is to be gained from such an exercise is a matter of speculation at this moment, some evidence in its favour may be gleaned from an application of the continuous dimension theory of type $II_{\infty}$ von Neumann algebras to finite difference equations given by Schaeffer in \cite{schaeffer}.

As a proof of concept, we provide an affirmative answer to Question 3 and leave the first two questions to the reader.

\begin{thm}
\textsl{
Let $\kM$ be a properly infinite von Neumann algebra acting on the Hilbert space $\sH$. Then there is a pair of non-singular positive self-adjoint operators $A, B$ affiliated with $\kM$ with their domains satisfying, $$\dom(A) \;\cap\; \dom(B) = \dom (A^{-1}) \;\cap\; \dom (B^{-1}) = \{ 0 \}_{\sH}.$$
}
\end{thm}
\begin{proof}
By \cite[Proposition 13]{kosaki_domains}, there is a pair of non-singular positive self-adjoint operators $A', B'$ on $\ell ^2(\N)$ satisfying the above-mentioned domain conditions, that is,
$$\dom(A') \;\cap\; \dom(B') = \dom (A'^{-1}) \;\cap\; \dom (B'^{-1}) = \{ 0 \}_{\ell ^2(\N) }.$$

Let $\Phi : \sB(\ell ^2(\N)) \to \kM$ be a unital normal embedding (say, as in Proposition \ref{prop:bh_embed}), and $\Phi_{\aff} : \sB(\ell ^2(\N))_{\aff}^c \to \afm ^c$ be the canonical extension.  From Theorem \ref{thm:phi_monoid_morph}-(iii), note that $\Phi_{\aff}(A'), \Phi_{\aff}(B')$ are non-singular, and $$\Phi_{\aff}(A'^{-1}) = \Phi_{\aff}(A')^{-1},\; \Phi_{\aff}(B'^{-1}) = \Phi_{\aff}(B')^{-1}.$$ By Theorem \ref{thm:sym_acc_sect}-(ii)-(v) and Theorem \ref{thm:phi_pres_closed}-(iii), we observe that $\Phi_{\aff}(A), \Phi_{\aff}(B)$ are positive self-adjoint operators in $\afm ^c$. From Theorem \ref{thm:phi_pres_closed}-(i) and Theorem \ref{thm:lat_morph}-(i), we note that, 
\begin{align*}
\dom \big( \Phi_{\aff}(A') \big) \cap \dom \big( \Phi_{\aff}(B') \big) &= \Phi_{\aff}^s \big( \dom (A') \cap \dom(B') \big) \\
&= \Phi_{\aff}^s(\{ 0 \}_{\ell^2(\N)})  = \{0\}_{\sH},\\
\dom \big( \Phi_{\aff}(A')^{-1} \big) \cap \dom \big( \Phi_{\aff}(B')^{-1} \big) &= \Phi_{\aff}^s \big( \dom (A'^{-1} ) \cap \dom \big( B'^{-1}) \big)\\
&= \Phi_{\aff}^s(\{ 0 \}_{\ell ^2(\N)}) = \{ 0 \}_{\sH}.
\end{align*}
Thus, $A = \Phi_{\aff}(A'), B = \Phi_{\aff}(B')$ satisfy the conditions stipulated in the assertion.
\end{proof}

\section{Acknowledgements}%
\label{sec:acknowledge}%

The first-named author is supported by a research fellowship provided by the Indian Statistical Institute. The second-named author is supported by Startup Research Grant (SRG/2021/002383) of SERB (Science and Engineering Research Board, India), and Startup-Grant by the Indian Statistical Institute (dated June 17, 2021). He would also like to express his gratitude to Charanya Ravi for helpful discussions. We would like to thank K V Krishna for sharing his PhD thesis on near-semirings and related references with us.%

\medskip

\bibliographystyle{amsalpha} 
\bibliography{references}%

\end{document}